\documentclass[10.5pt]{article}
\usepackage{color}
\usepackage{amsfonts}
\usepackage{graphicx}
\usepackage{amsmath,enumerate}
\usepackage{cases}
\usepackage{amsthm}
\usepackage{amssymb}
\usepackage{amsfonts}
\usepackage{latexsym,amssymb,amsfonts}
\usepackage{mathrsfs}
\usepackage{graphicx}
\usepackage{psfrag}
\usepackage{amsmath, amsbsy}
\usepackage{amsopn, amstext}
\usepackage{amsmath,multicol,amsopn}
\usepackage{latexsym,amssymb}
\usepackage{hyperref}

\usepackage{lineno}%
\def\dbB{{\mathbb{B}}}

\def\dbE{{\mathbb{E}}}
\def\dbF{{\mathbb{F}}}

\def\dbH{{\mathbb{H}}}
\def\dbI{{\mathbb{I}}}

\def\dbL{{\mathbb{L}}}

\def\dbP{{\mathbb{P}}}

\def\dbR{{\mathbb{R}}}
\def\dbS{{\mathbb{S}}}

\def\dbV{{\mathbb{V}}}

\def\d{\delta}
\def\e{\varepsilon}

\def\k{\kappa}

\def\t{\tau}
\def\f{\varphi}

\def\Th{\Theta}

\def\3n{\negthinspace \negthinspace \negthinspace }
\def\2n{\negthinspace \negthinspace }
\def\1n{\negthinspace }
\def\ns{\noalign{\smallskip} }

\def\ds{\displaystyle}
\def\G{\Gamma}
\def\D{\Delta}
\def\Th{\Theta}

\def\Om{\Omega}

\def\cA{{\cal A}}

\def\cC{{\cal C}}

\def\cF{{\cal F}}
\def\cG{{\cal G}}

\def\cJ{{\cal J}}

\def\cP{{\cal P}}
\def\cQ{{\cal Q}}
\def\cR{{\cal R}}

\def\cU{{\cal U}}

\def\mE{{\mathbb{E}}}

\def\ss{\smallskip}
\def\ms{\medskip}

\def\q{\quad}
\def\qq{\qquad}
\def\hb{\hbox}

\def\lan{\big\langle}
\def\ran{\big\rangle}

\def\esssup{\mathop{\rm esssup}}

\def\wt{\widetilde}
\def\cd{\cdot}

\def\deq{\mathop{\buildrel\D\over=}}

\def\({\Big (}
\def\){\Big )}
\def\[{\Big[}
\def\]{\Big]}

\def\={\buildrel \triangle \over =}

\def\ee{\end{equation}}
\def\bea{\begin{eqnarray}}
\def\eea{\end{eqnarray}}
\def\bt{\begin{theorem}}
\def\et{\end{theorem}}
\def\bc{\begin{corollary}}
\def\ec{\end{corollary}}
\def\bl{\begin{lemma}}
\def\el{\end{lemma}}
\def\bp{\begin{proposition}}
\def\ep{\end{proposition}}
\def\br{\begin{remark}}
\def\er{\end{remark}}
\def\ba{\begin{array}}
\def\ea{\end{array}}
\def\bde{\begin{definition}}
\def\ede{\end{definition}}

\newtheorem{lemma}{Lemma}[section]
\newtheorem{remark}{Remark}[section]
\newtheorem{example}{Example}[section]
\newtheorem{theorem}{Theorem}[section]
\newtheorem{corollary}{Corollary}[section]

\newtheorem{definition}{Definition}[section]
\newtheorem{proposition}{Proposition}[section]

\def\punct{}
\newtheoremstyle{dotless}{}{}{\rm}{}{\bf}{\punct}{.5em}{}
\theoremstyle{dotless}

\newtheorem{taggedassumptionx}{}
\newenvironment{taggedassumption}[1]
{\renewcommand\thetaggedassumptionx{#1}\taggedassumptionx}
{\endtaggedassumptionx}

\allowdisplaybreaks[4]
\makeatletter

\@addtoreset{equation}{section}
\makeatother

\title{\bf
Forward-Backward Stochastic Linear-Quadratic Optimal Controls: Equilibrium Strategies and Non-Symmetric  Riccati Equations
}

\author{
Qi L\"{u}\thanks{School
of Mathematics, Sichuan University, Chengdu
610064, China. (Email: {\tt lu@scu.edu.cn}). Qi L\"{u} is supported by the NSF of China under grant 12025105.}
~~~~
Bowen Ma\thanks{School of Mathematical Sciences, Chengdu University of Technology, Chengdu, 610059, China
(Email: {\tt albertmabowen@gmail.com}). }
~~~~
Hanxiao Wang\thanks{School of Mathematical Sciences, Shenzhen University, Shenzhen,
518060, China (Email: {\tt hxwang@szu.edu.cn}). Hanxiao Wang is supported in part by NSFC Grant 12201424,
Guangdong Basic and Applied Basic Research Foundation 2023A1515012104,
the Science and Technology Program of Shenzhen RCBS20231211090537064,
and Shenzhen University 2035 Program for Excellent Research Grant 2024C008.
}
}

\date{\today}

\begin{document}
\maketitle
\begin{abstract}
Linear-quadratic  optimal control problem for systems governed by forward-backward stochastic differential equations has been extensively studied over the past three decades. Recent research has revealed that for forward-backward control systems, the corresponding optimal control problem is inherently time-inconsistent. Consequently, the optimal controls derived in existing literature represent pre-committed solutions rather than dynamically consistent strategies.  In this paper, we shift focus from pre-committed solutions to addressing the time-inconsistency issue directly, adopting a dynamic game-theoretic approach to derive equilibrium strategies. Owing to the forward-backward structure, the associated equilibrium Riccati equation (ERE) constitutes a coupled system of matrix-valued, non-local ordinary differential equations with a non-symmetric structure. This non-symmetry introduces fundamental challenges in establishing the solvability of the EREs.
We overcome the difficulty by establishing a priori estimates for a combination of the solutions to EREs,
which,  interestingly, is a representation of the equilibrium value function.
\end{abstract}

\section{Introduction}\label{sec-intro}

Let $W(\cd)$ be a standard one-dimensional Brownian motion defined on a complete filtered probability space $(\Omega,\cF,\dbF,\dbP)$, where $\dbF=\{\cF_t\}_{t\geq 0}$ is the augmented natural filtration generated by $W(\cd)$.

Fix  $T>0$, and denote by $\dbS^n$ the subspace of $\dbR^{n\times n}$ consisting of symmetric matrices, and $\dbS_+^n$ the subset of $\dbS^n$ comprising positive semi-definite matrices.

For any initial time $t\in [0,T)$  and state $x$ in $L_{\cF_t}^2(\Omega;\dbR^n)$ (the space of $\cF_t$-measurable, square-integrable random variables), consider the following controlled forward-backward stochastic differential equation (FBSDE) on $[t,T]$:\vspace{-1mm}
\begin{equation}\label{state}
\begin{cases}
dX(s)=\big[A(s)X(s)+B(s)u(s)\big]ds+\big[C(s)X(s)+D(s)u(s)\big]dW(s),\\
dY(s)=-\big[\widehat{A}(s)X(s)+\widehat{B}(s)u(s)+\widehat{C}(s)Y(s) +\widehat{D}(s)Z(s)\big]ds+Z(s)dW(s), \\
X(t)=x,\q Y(T)=HX(T),
\end{cases}
\end{equation}
where  $A(\cd),C(\cd):[0,T]\to  \dbR^{n\times n}$, $B(\cd),D(\cd):[0,T]\to  \dbR^{n\times k}$, $\widehat{C}(\cd),\widehat{D}(\cd):[0,T]\to \dbR^{m\times m}$, $\widehat{A}(\cd):[0,T]\to \dbR^{m\times n}$, and
$\widehat{B}(\cd):[0,T]\to \dbR^{m\times k}$ are given bounded deterministic functions,  $H\in \dbR^{m\times n}$,
and $u(\cd)$ is the control process in $\cU[t,T]$, the space of $\dbR^k$-valued, square-integrable, $\dbF$-adapted processes on $ [t,T]\times \Omega$.

The associated cost functional is given by\vspace{-1mm}
\begin{align}
\cJ^0(t,x;u(\cd))
&=\frac{1}{2}\mE_t\Big[ \int_{t}^{T} \Big(\langle Q(s)X(s),X(s) \rangle +\langle R(s)u(s),u(s) \rangle+ \langle M(s)Y(s),Y(s) \rangle\nonumber\\
& \qq+ \langle N(s)Z(s),Z(s) \rangle \Big)ds+  \langle G_1X(T),X(T) \rangle + \langle G_2Y(t),Y(t) \rangle \Big],\label{cost}
\end{align}
where $Q(\cd): [0,T]\to \dbS^n$, $R(\cd):[0,T]\to \dbS^k$, and
$M(\cd),N(\cd):[0,T]\to \dbS^m$ are bounded matrix-valued functions,  $G_1\in \dbS^n$, and $G_2\in \dbS^m$.
The optimal control problem for \eqref{state} and \eqref{cost} can be stated as follows:

\ms
\noindent
{\bf Problem (FBSLQ).} For any given initial pair  $(t,x)\in [0,T)\times L_{\cF_t}^2(\Omega;\dbR^n)$,
find an optimal control $u^*(\cd) \in \cU[t,T]$ such that\vspace{-1mm}
\begin{equation}\label{J0-optimal}
\cJ^0(t,x;u^*(\cd))=\inf_{u(\cd) \in \cU[t,T]} \cJ^0(t,x;u(\cd))=V^0(t,x).
\end{equation}

The aforementioned problem is usually referred to as a linear-quadratic  optimal control (LQ, for short) problem for forward-backward stochastic differential equations.
A control process $u^*(\cd)\in \cU[t,T]$  satisfying \eqref{J0-optimal} is called an optimal control for the given initial pair $(t,x)$.

LQ problem for forward stochastic differential equations (SDEs) was first investigated by Wonham \cite{Wonham-1968} in the 1960s, and has since remained a cornerstone of stochastic control theory. While our current work focuses on forward-backward systems, we refer readers to several classical texts \cite{Yong-Zhou1999, Bensoussan2018, Sun-Yong-2020} for comprehensive surveys, tutorials, and extensive references on this foundational topic.

The solvability of general nonlinear backward stochastic differential equations (BSDEs) was first established
by Pardoux and Peng \cite{Pardoux-Peng-1990}, marking a breakthrough that established BSDEs as a fundamental tool in stochastic control theory and mathematical finance. This development naturally motivated the study of optimal control problems for systems governed by BSDEs or FBSDEs.

The study of optimal control problems for FBSDEs holds significant theoretical and practical importance, particularly in financial applications involving investment problems with specified future conditions (modeled as random variables) \cite{Duffie-Geoffard-Skiadas-1994,Karoui-Peng-Quenez-1997,Lim-Zhou-2001,Lim2004}. These problems also play an important role in stochastic differential game theory, where forward-backward LQ optimal control problems naturally emerge in leader-follower game formulations \cite{Bensoussan2018,Yong-2002}.

Due to these fundamental connections, LQ problems for FBSDEs have attracted considerable research attention. Important contributions include: Lim's application of backward LQ control theory to quadratic hedging and mean-variance portfolio selection \cite{Lim2004}; Huang, Li, and Wang's work on near-optimal control \cite{Huang-Li-Wang-2010}; Wang, Wu, and Xiong's study of partial information problems \cite{Wang-Wu-Xiong-2018}; Sun, Wu, and Xiong's extension to indefinite cost functionals \cite{Sun-Wu-Xiong-2023}; Sun, Wang, and Wen's investigations of Stackelberg games \cite{Sun-Wang-Wen2021}; and Hu, Ji, and Xue's comprehensive analysis of fully coupled FBSDEs \cite{Hu-Ji-Xue-2023}.

Despite these advances, a fundamental challenge remains unresolved: How to establish the solvability of the Riccati equation associated with Problem (FBSLQ)? This persists as an open problem even under the strong assumption of uniformly positive definite weighting matrices in the cost functional.

Recent work by Wang, Yong, and Zhou \cite{Wang-Yong-Zhou-2024} has revealed that the forward-backward structure of the state equation inherently induces time-inconsistency in the corresponding optimal control problem. This  discovery establishes that Problem (FBSLQ) is generically time-inconsistent. For concrete illustrations of this phenomenon, we refer to \cite[Example 1.1]{Wang-Yong-Zhou-2024} and \cite[Example 1.2]{Lu-2024}, which provide explicit counterexamples to demonstrate this time-inconsistency.

In practical applications, time inconsistency poses a significant challenge: the optimal control derived at the initial time may lose its optimality as the system evolves, necessitating continuous strategy adjustments by decision-makers.
The earliest mathematical consideration of time-inconsistent optimal controls was given by Strotz
\cite{Strotz-1955}, followed by the recent works of
Ekeland and Pirvu \cite{Ekeland2008}, Ekeland and Lazrak \cite{Ekeland2010},
Bj\"ork, Khapko, and Murgoci  \cite{Bjork-2017}, and Yong \cite{Yong-2014,Yong-2017} for various kinds of problems,
to mention a few.
For the time-inconsistent stochastic LQ problems,
Hu, Jin, and Zhou \cite{Hu-2012,Hu-2017} considered the equilibrium solution in the open-loop sense;
Yong \cite{Yong-2017} studied the closed-loop equilibrium strategy by introducing the multi-person differential game method;
Dou and L\"{u} \cite{Dou-2020} extended the result of \cite{Yong-2017} to the infinite-dimensional system described by a stochastic evolution equation;
Lazrak, Wang, and Yong \cite{Lazrak-WY2023} studied a special  zero-sum game problem motivated by some financial applications;
Cai et al. \cite{Cai-2022}  gave some equivalent characterizations of the existence of
a closed-loop equilibrium strategy for ODEs.

Further, to account for subjective time preference--another fundamental source of time-inconsistency--we introduce the following generalized cost functional extending \eqref{cost}:\vspace{-2mm}
\begin{align}
\cJ(t,x;u(\cd))
&=\frac{1}{2}\mE_t\Big[ \int_{t}^{T} \Big(\langle Q(t,s)X(s),X(s) \rangle +\langle R(t,s)u(s),u(s) \rangle+ \langle M(t,s)Y(s),Y(s) \rangle\nonumber\\
& \qq+ \langle N(t,s)Z(s),Z(s) \rangle \Big)ds+  \langle G_1(t)X(T),X(T) \rangle + \langle G_2(t)Y(t),Y(t) \rangle \Big].\label{TI-cost}
\end{align}
We refer to the control problem with state equation \eqref{state} and cost functional \eqref{TI-cost} as \textbf{Problem (TI-FBSLQ)}.

To continue, we introduce some notations that will be used consistently throughout this paper.
For any matrices $M\in \dbR^{n\times m}$, let $|M|$ denote its  spectral norm,
which equals the square root of the largest eigenvalue of $M^\top M$,
and $|M|_{\rm Tr}$  its trace norm,
which equals the sum of the square roots of the eigenvalues of $M^\top M$.
Furthermore, for time interval $0\leq t_1\leq t_2\leq T$, we introduce the triangular domain\vspace{-1mm}
$$
\D^*[t_1,t_2]=\{(t,s)\in[t_1,t_2]^2\hbox{ with } t\leq s\}.
$$
For an arbitrary Euclidean space $\dbH$, we introduce the following function spaces:\vspace{-1mm}
\begin{align*}
& L^2(0,T;\,\dbH)= \Big\{\f:[0,T]\to\dbH\bigm|\f(\cd)~\hb{is measurable, }\int_0^T|\f(s)|^2ds<\infty \Big\};\\
& L^\infty(0,T;\,\dbH)= \Big\{\f:[0,T]\to\dbH\bigm|\f(\cd)~\hb{is measurable, }\esssup_{0\leq s\leq T}|\f(s)|<\infty \Big\};\\
& C([0,T];\,\dbH)= \Big\{\f:[0,T]\to\dbH\bigm|\f(\cd)~\hb{is continuous, }\sup_{0\leq s\leq T}|\f(s)|<\infty \Big\};\\
& C(\D^*[0,T];\,\dbH)= \Big\{\f:\D^*[0,T]\to\dbH\bigm|\f(\cd,\cd)~\hb{is continuous, }\sup_{0\leq t\leq s\leq T}|\f(t,s)|<\infty \Big\};\\
&C^{\infty}([0,T];\,\dbH)= \Big\{\f:[0,T]\to\dbH\bigm|\f(\cd)~\hb{is infinitely differentiable}\Big\};\\
& L_{\cF_t}^2(\Om;\,\dbH)
=\Big\{\f:\Om\to\dbH\bigm|\f(\cd)~\hb{is $\cF_t$-measurable},\q \dbE\big[|\f(\omega)|^2\big]<\infty \Big\};\\
& L_\dbF^2(\Om;C([0,T];\dbH))
=\Big\{\f:[0,T]\times\Om\to\dbH\bigm|\f(\cd)~\hb{is $\dbF$-adapted, pathwise continuous, }  \dbE\[\ds\sup_{0\leq s\leq T}|\f(s)|^2\]<\infty \Big\};\\
& L_\dbF^2(0,T;\,\dbH)
=\Big\{\f:[0,T]\times\Om\to\dbH\bigm|\f(\cd)~\hb{is $\dbF$-progressively measurable, }   \dbE\[\int_0^T|\f(s)|^2ds\]<\infty \Big\}.
\end{align*}

Next, we introduce the following assumptions:\vspace{-1mm}
\begin{taggedassumption}{(H1)}\label{H1}
The system coefficients and weighting matrices satisfy the following regularity conditions:
\begin{align}
&A(\cd),C(\cd) \in L^\infty(0,T;\dbR^{n\times n}),\q B(\cd),D(\cd) \in L^\infty(0,T;\dbR^{n\times k}), \nonumber\\
&\widehat{A}(\cd)\in L^\infty(0,T;\dbR^{m\times n}),\q \widehat{B}(\cd)\in L^\infty(0,T;\dbR^{m\times k}),\q\widehat{C}(\cd),\widehat{D}(\cd)\in L^\infty(0,T;\dbR^{m\times m}), \nonumber\\&
Q(\cd,\cd)\in C([0,T]^2;\dbS^n),\q R(\cd,\cd)\in C([0,T]^2;\dbS^k),\q M(\cd,\cd),N(\cd,\cd)\in C([0,T]^2;\dbS^m),\nonumber \\&
G_1(\cd)\in C([0,T];\dbS^n),\q G_2(\cd)\in C([0,T];\dbS^m).
\end{align}
\end{taggedassumption}

Under Assumption \ref{H1}, the classical well-posedness results of SDEs and BSDEs (see, e.g.,\cite[Sections 3.1 and 4.1]{Lu-Zhang-2021}) guarantees that
for any initial pair  $(t,x)\in [0,T)\times L_{\cF_t}^2(\Omega;\dbR^n)$ and control $u(\cd) \in \cU[t,T]= L_{\dbF}^2(t,T;\dbR^{k})$,
the state equation \eqref{state} admits a unique solution $\big( X(\cd),Y(\cd),Z(\cd)  \big)\in L^2_{\dbF}(\Omega;C([0,T];\dbR^n))\times
L^2_{\dbF}(\Omega;C([0,T];\dbR^m)) \times L^2_{\dbF}(0,T;\dbR^m)$.
Then the cost functional \eqref{TI-cost} is well-defined.\vspace{-1mm}

\begin{definition}\label{de1}\rm
A matrix-valued function ${\Theta}(\cdot)\in L^{\infty} (0,T;\mathbb{R}^{k\times n})$ is called a {\it closed-loop equilibrium strategy}  for Problem (TI-FBSLQ) if for every initial pair $(t,x)\in [0,T)\times L_{\mathcal{F}_t}^2(\Omega;\mathbb{R}^n)$ and any perturbation
$v\in L_{\mathcal{F}_t}^2(\Omega;\mathbb{R}^k)$, the following variational inequality holds:\vspace{-1mm}
\begin{equation}
\varliminf_{\e \to 0} \frac{\cJ\left(t, \bar{X}(t) ;
u^{\varepsilon}(\cdot)\right)-\cJ\left(t, \bar{X}(t) ;
\bar{u}(\cdot)\right)}{\e} \geq 0,\qquad
\mathbb{P}\mbox{\rm-a.s.,}
\end{equation}
where the reference trajectory $(\bar{X}(\cd), \bar{Y}(\cd), \bar{Z}(\cd))$ satisfies:\vspace{-1mm}
\begin{equation}\label{de1-eq2}
\begin{cases}
\ds	d\bar{X}(s)=\big[A(s)+B(s)\Theta(s)\big]\bar{X}(s)ds+\big[C(s)+D(s)\Theta(s)\big]\bar{X}(s)dW(s), & s\in [t,T], \\
\ns\ds	d\bar{Y}(s)=-\big[ \big(\widehat{A}(s) + \widehat{B}(s)\Theta(s)\big)\bar{X}(s)+\widehat{C}(s)\bar{Y}(s) +\widehat{D}(s)\bar{Z}(s)\big]ds+\bar{Z}(s)dW(s), & s\in [t,T],\\
\ns\ds	\bar{X}(t)=x,\q \bar{Y}(T)=H\bar{X}(T),
\end{cases}
\end{equation}
the control processes are given by \vspace{-1mm}
\begin{equation}\label{de1-eq1}
\bar{u}(s)=\Theta(s)\bar{X}(s),\q
u^{\e}(s)=v\chi_{[t,t+\e]}(s)+{\Theta}(s){X}^{\e}(s),
\end{equation}
and $ X^{\e}(\cd)$ denotes the solution to the state equation \eqref{state} corresponding to the perturbed control $u^{\e}(\cdot)$.
\end{definition}

In the rest of this paper, for notational simplicity, given any ${\Theta}(\cdot)\in L^{\infty} (0,T;\mathbb{R}^{k\times n})$, we denote\vspace{-2mm}
\begin{equation}\label{Atheata}
	\begin{aligned}
		A_\Theta(s)\=A(s)+B(s)\Theta(s) ,\q C_\Theta(s)\= C(s)+D(s)\Th(s),\q \widehat{A}_\Theta(s)\=\widehat{A}(s)+\widehat{B}(s)\Th(s),\q s\in[0,T].
	\end{aligned}
\end{equation}
\begin{remark}
The concept of closed-loop equilibrium strategies follows from the framework established by Bj\"ork,  Khapko, and  Murgoci \cite{Bjork-2017}. Within our linear-quadratic setting, we may restrict our attention to linear feedback strategies without loss of generality. This approach aligns with the economic intuition developed in \cite{Bjork-2017,Hu-2012,Yong-2017}, where at each decision time $t$, the controller engages in a cooperative game with their future selves. Specifically:\vspace{-1mm}
\begin{itemize}
\item The controller optimizes the cost functional over an infinitesimal time interval $[t,t+\e)$.
\item The optimization is performed under the constraint that the control is no more optimal beyond $t+\e$.
\item This leads to a time-consistent strategy that accounts for future decision-making processes.
\end{itemize}
\end{remark}

To establish the existence of closed-loop equilibrium strategies for Problem (TI-FBSLQ),
inspired by \cite[Theorem 1.6]{Lu-2024}, we introduce
the following equilibrium Riccati equation (ERE, for short):
\begin{equation}\label{ERE}
\begin{cases}\ds
\frac{d P_1(t,s)}{ds}+ P_{1}(t,s) A_\Theta(s)  +A_\Theta(s)^\top  P_1(t,s) +C_\Theta(s)^\top  P_{1}(t,s)C_\Theta(s)  \\\ns\ds\q
+Q(t,s)+\Th(s)^\top R(t,s)\Th(s) +P_2(s)^\top M(t,s)P_2(s) \\\ns\ds\q
+ C_\Theta(s)^\top P_2(s)^\top N(t,s) P_2(s) C_\Theta(s)=0, & 0\leq t\leq s\leq T,\\\ns\ds
\frac{dP_2(s)}{ds}+ P_2(s)A_\Theta(s)+  \widehat{A}_\Theta(s)  +\widehat{C}(s) P_2(s) +\widehat{D}(s) P_2(s)C_\Theta(s) =0, & 0\leq s\leq T,\\\ns\ds
P_1(t,T)=G_1(t),	\q P_2(T)=H, & 0\leq t\leq T,
\end{cases}
\end{equation}
where \vspace{-3mm}
\begin{align}\label{Theta}
\Theta(s)&=-\big[R(s,s)+D(s)^\top\big(P_1(s,s) + P_2(s)^\top N(s,s)P_2(s)\big)D(s)\big]^{-1} \nonumber  \\
&\q\times \big[B(s)^\top P_1(s,s)+D(s)^\top\big(P_1(s,s)+P_2(s)^\top N(s,s)P_2(s)\big)C(s)\nonumber\\
&\qq+ \big(\widehat{B}(s)^\top+B(s)^\top P_{2}(s)^\top + D(s)^\top P_{2}(s)^\top  \widehat{D}(s)^\top \big) G_2(s)P_{2}(s) \big],\qq 0\leq s\leq T.
\end{align}

Before studying the well-posedness of ERE \eqref{ERE}, we first examine an illustrative example that demonstrates how the highly nonlinear nature of $P_2(\cd)$ can result in the non-existence of solutions to \eqref{ERE}.

\begin{example}\label{exam1}
\rm
Let us examine the ERE  \eqref{ERE} with coefficients:
$
\widehat{A}=\widehat{B}=\widehat{C}=\widehat{D}=A=C=D=Q=N=G_2=0,
B=R=H=G_1=1,\, M=2e,\,T=2.
$
Under these conditions, the ERE \eqref{ERE} simplifies to \vspace{-2mm}
\begin{align*}
\begin{cases}
\dot{P}_1(s)= P_1(s)^2-2e P_2(s)^2,& s\in [0,2],\\
\dot{P}_2(s)=P_2(s)P_1(s),& s\in [0,2],\\
P_1(2)=1,\q P_2(2)=1.
\end{cases}
\end{align*}
Through direct computation, we obtain $P_2(s)=e^{-\int_{s}^{2}P_1(\tau)d\tau }$ and $ \dot{P}_1(s)=P_1(s)^2-2e^{P_1(s)^2}$.
Consider the auxiliary equation:\vspace{-2mm}
\begin{equation*}
\begin{cases}
\dot{P}(s)=-P(s)^2,\q s\in (1,2],\\
P(2)=1,
\end{cases}
\end{equation*}
which admits the explicit solution $P(s)=\frac{1}{s-1}$ for $s\in (1,2]$.
Now, suppose by contradiction that $(P_1(\cd),P_2(\cd))$ exists globally on $[0,2]$.
Noting the inequality $x^2-2e^{x^2}< -x^2$ and applying the comparison principle, we deduce that $P(s)<P_1(s)$ for $s\in (1,2)$. However, this implies $\lim_{s\to 1+}P_1(s)=+\infty$,
yielding the desired contradiction to the assumed global existence.
\end{example}

Motivated by Example \ref{exam1}, we introduce the following assumption:

\begin{taggedassumption}{(H2)}\label{H2}
There exists a constant $\delta>0$, such that for any $0\leq t\leq s \leq T$,\vspace{-2mm}
\begin{equation*}
R(t,s)\geq \delta I,\, G_2(t)\geq \delta I, \text{and } Q(t,s)\geq 0,\,M(t,s)\geq 0,\,N(t,s)\geq 0,\, G_1(t)\geq 0.
\end{equation*}
\end{taggedassumption}
\begin{taggedassumption}{(H3)}\label{H3}
For any $0\leq t\leq \tau \leq s\leq  T$, the following monotonic condition holds:\vspace{-2mm}
\begin{align*}
&Q(t,s)\leq Q(\tau,s) \leq \widehat{Q}, \q R(t,s)\leq R(\tau,s)\leq \widehat{R}, \q M(t,s)\leq M(\tau,s) \leq \widehat{M}, \nonumber\\&
N(t,s)\leq N(\tau,s)\leq \widehat{N},\q  G_1(t)\leq G_1(\tau) \leq \widehat{G}_1,\q G_2(t) \leq G_2(\tau)\leq  \widehat{G}_2.\vspace{-2mm}
\end{align*}
where $\widehat{Q},\widehat{R},\widehat{M},\widehat{N}, \widehat{G}_1, \widehat{G}_2$ are given positive definite matrices.
\end{taggedassumption}

\begin{remark}
Typical examples satisfying Assumption \ref{H3} include:
\begin{align*}
Q(t,s)=\lambda(s-t)Q(s),\q R(t,s)=\lambda(s-t)R(s),
\end{align*}
where $\lambda(\cd)$ represents the discounting function, and $Q(\cd)$ and $R(\cd)$ are  positive definite.
Such monotonic conditions are also assumed in \cite{Yong-2017}.
\end{remark}

The following result shows the equivalence between the existence of a closed-loop equilibrium strategy and the well-posedness of ERE \eqref{ERE}.
This result can be regarded as the verification theorem of Problem (TI-FBSLQ).
It is a direct consequence of our previous work \cite[Theorem 1.6]{Lu-2024}, and we put it here for completeness and the convenience of readers.

\begin{proposition}\label{Th1}
Let Assumptions \ref{H1}--\ref{H3} hold.
Then Problem (TI-FBSLQ) admits a closed-loop equilibrium strategy if and only if
the ERE \eqref{ERE} admits a solution $(P_1(\cd,\cd),P_2(\cd))\in C(\Delta^*[0,T];\dbS^n_+)\times C([0,T];\dbR^{m\times n})$.
Moreover, in this case, the function $\Theta(\cd)$ defined by \eqref{Theta} is a closed-loop equilibrium strategy.
\end{proposition}
\vspace{-2mm}
\begin{remark}
Assumptions \ref{H2} and \ref{H3} are  imposed to ensure the well-posedness of the equilibrium Riccati equation. However, this requirement can be relaxed when focusing solely on the verification theorem for Problem (TI-FBSLQ), as demonstrated in \cite{Lu-2024}.
\end{remark}

From Proposition \ref{Th1}, we see that the central challenge in constructing a closed-loop equilibrium strategy for Problem (TI-FBSLQ) reduces to establishing the solvability of \eqref{ERE}, which represents the primary focus of the current work.

Although we introduce Assumptions \ref{H1}--\ref{H3} to study the ERE \eqref{ERE}, establishing the well-posedness of ERE \eqref{ERE} is still nontrivial. This complexity arises from several distinctive features that differentiate ERE \eqref{ERE} from classical stochastic LQ Riccati equations (see e.g., \cite{Sun-Yong-2020}):

\ss

1. {\bf Non-local Temporal Dependence}.
The time-preference in the weighting matrices causes $P_1(\cd,\cd)$ to depend on two time variables $(t,s)$,
making it non-local in the first time variable $t$. Alternatively,
if we treat $t$ as the index of a component,
then the ERE \eqref{ERE} above can be viewed as an (uncountably) infinite dimensional systems of
matrix-valued ordinary differential equations (ODEs),
which is self interacted via the diagonal term $P_1(s,s)$.

\ss

2. {\bf Non-symmetric Coupled Structure}.
The ERE \eqref{ERE} consists of two coupled matrix-valued ODEs where $P_2(\cd)$ breaks the symmetry $P_2(\cd)\equiv P_2(\cd)^\top$.
Moreover, as mentioned before, Problem (FBSLQ) (in which the cost functional reduces to \eqref{cost}) remains time-inconsistent.
In this case,
the corresponding ERE \eqref{ERE} reduces to the following coupled ODEs:
\begin{equation*}
\begin{cases}\ds
\frac{d P_1(s)}{ds}+ P_{1}(s) A_\Theta(s)  +A_\Theta(s)^\top  P_1(s) +C_\Theta(s)^\top  P_{1}(s)C_\Theta(s) +Q(s)+\Th(s)^\top R(s)\Th(s)  \\\ns\ds\q
+P_2(s)^\top M(s)P_2(s)
+ C_\Theta(s)^\top P_2(s)^\top N(s) P_2(s) C_\Theta(s)=0, & 0\leq s\leq T,\\\ns\ds
\frac{dP_2(s)}{ds}+ P_2(s)A_\Theta(s)+  \widehat{A}_\Theta(s)  +\widehat{C}(s) P_2(s) +\widehat{D}(s) P_2(s)C_\Theta(s) =0, & 0\leq s\leq T,\\\ns\ds
P_1(T)=G_1,	\q P_2(T)=H,
\end{cases}
\end{equation*}
with\vspace{-2mm}
\begin{align*}
\Theta(s)&\!=\!-\big[R(s)\!+\!D(s)^\top\!\!\big(P_1(s)\! + \!  P_2(s)^\top\! N(s)P_2(s)\big)D(s)\big]^{-1}\! \big[B(s)^\top\! P_1(s)\!+\!D(s)^\top\big(P_1(s)\!+\!P_2(s)^\top\! N(s)P_2(s)\big)C(s)\nonumber\\
&\qq+ \big(\widehat{B}(s)^\top+ {B}(s)^\top P_{2}(s)^\top  + D(s)^\top P_{2}(s)^\top  \widehat{D}(s)^\top \big) G_2(s)P_{2}(s) \big],\qq 0\leq s\leq T.
\end{align*}
Thus, the inherent non-symmetry of \eqref{ERE} stems from the forward-backward structure of the state equations, representing a significant departure from classical LQ theory.

\ss

From the above, we know that the essence of Problem (TI-FBSLQ) in our work lies in two aspects:
First, the forward-backward structure of the state equation,
which brings the complicated coupled relationship and non-symmetric structure in ERE \eqref{ERE}.
This together with the high nonlinearity of ERE \eqref{ERE} constitutes the most fundamental challenge of this work;
Second, the time-preference for the weighting matrices,
which makes ERE \eqref{ERE} non-local in the first time variable,
or in another perspective, makes ERE \eqref{ERE} an  (uncountably) infinite dimensional systems of matrix-valued ODEs.

\begin{remark}
When restricted to forward SDE systems, the ERE \eqref{ERE} reduces to a non-local but symmetric matrix-valued ODE. For such cases, Yong \cite{Yong-2017} pioneered a solution approach via multi-person differential games. However, this method relies critically on:

1. The dynamic programming principle (DPP) of forward LQ systems;

2. Complete solvability of the pre-commitment Riccati equation.

Both conditions fail for our FBSDE framework, even in the simplified cost case \eqref{cost}, rendering Yong's methodology inapplicable.
\end{remark}
\begin{remark}
It is worthy of pointing out that the Riccati-type ODEs with non-symmetric structures
arise naturally in many important settings, such as  the study of open-loop time-inconsistent LQ control problems \cite{Hu-2012,Hu-2017},
non-zero LQ games \cite{Sun-Yong-2019}, and LQ mean-field games \cite[Chapter 6]{Bensoussan2013}.
However, establishing general existence results for such systems remains a major open challenge in stochastic control theory, with currently limited constructive solution approaches available.
\end{remark}

In this paper, we demonstrates that despite the inherent non-symmetry of  ERE  \eqref{ERE}, careful analysis establishes the existence of a unique solution. Our approach overcomes the fundamental challenges by treating the solution components $P_1(\cd,\cd)$ and $P_2(\cd)$ as an interconnected system rather than separate entities. The core methodology consists of three key steps:
\begin{itemize}
\item[(i)] A priori estimates for the combined quantity $P_1(t,t)+P_2(t)^\top G_2(t)P_2(t)$ are established through direct differentiation, under sufficient smoothness conditions on the system coefficients and weighting matrices.

\item[(ii)]
These estimates enable the construction of suitable Picard iteration sequences, forming the foundation for proving both local and global well-posedness of the ERE \eqref{ERE}.

\item[(iii)]
The results are extended to non-smooth cases via a mollification approximation technique.
\end{itemize}
The most challenging part of our approach lies in identifying an appropriate invariant
$P_1(t,t)+P_2(t)^\top G_2(t)P_2(t)$ for establishing a priori bounds.
A key point is that we observe that the function $\frac{1}{2}\lan \big(P_1(t,t)+P_2(t)^\top\! G_2(t)P_2(t)\big) x$, $x\ran;(t,x)\in[0,T]\times\dbR^n$
precisely represents the equilibrium value function of Problem (TI-FBSLQ). The derivation of these estimates requires meticulous analysis and remains non-trivial despite our structural insights.

In summary, this work addresses the unavoidable time-inconsistency inherent in Problem (TI-FBSLQ)--a consequence of its forward-backward system structure--by:
\begin{itemize}
\item Developing a framework for closed-loop equilibrium strategies;
\item Introducing novel techniques to handle the non-symmetric ERE structure;
\item Providing constructive methods for equilibrium strategy derivation.
\end{itemize}

Now we are ready to present main results of this paper.

\begin{theorem}\label{Th2}
Let Assumptions \ref{H1}--\ref{H3} hold. Then the ERE \eqref{ERE} admits
a unique solution $(P_1(\cd,\cd),$ $P_2(\cd))\in C(\Delta^*[0,T];\dbS^n_+)\times C([0,T];\dbR^{m\times n})$.
Furthermore, the combined quantity $P_1(t,t)+P_2(t)^\top G_2(t)P_2(t)$ admits the following uniform estimate:
\begin{equation}
\sup_{t\in[0,T]}\big|P_1(t,t)+P_2(t)^\top G_2(t)P_2(t) \big|\leq \cC^*_{\dbV},
\end{equation}
where the bounding constant $\cC^*_{\dbV}$ is explicitly given by
\begin{align}\label{def-CV}
& \cC^*_{\dbV}\= \! \Big[ \widehat{\cG}\!+\!\widehat{\cG}|H|^2\!+\! \int_{0}^{T} \(|\widehat{Q}|\!+\! \widehat{\cG}^2|\widehat{A}(s)|^2\) ds  \Big] e^{\int_{0}^{T}\big[ 2|A(s)|+|C(s)|^2  + \frac{n}{\delta} \big( 1+ |\widehat{M}| +|\widehat{N}||C(s)|^2 + 2\widehat{\cG}|\widehat{C}(s)|+ 2 \widehat{\cG}|C(s)||\widehat{D}(s)|  \big)\big]ds }.
\end{align}
Here, we define $\widehat{\cG}\=\max\{|\widehat{G}_1|,|\widehat{G}_2|\}$, and  $\delta$, $\widehat{Q}$, $\widehat{M}$, $\widehat{N}$ are given in Assumptions \ref{H2}--\ref{H3}.
\end{theorem}

\begin{remark}
In \cite{Lu-2024}, the authors established the solvability of the ERE \eqref{ERE} in the one-dimensional case under the  assumption that the diffusion coefficient satisfies a uniform lower bound, i.e., $D(s)^2>\d>0$ for all $s\in[0,T]$. These assumptions effectively impose linear growth conditions on the ERE \eqref{ERE} and eliminate its inherent non-symmetric structure, thereby circumventing the principal analytical challenges. Consequently, the main challenge is avoided, and the existence of solutions to \eqref{ERE} could be obtained through standard fixed-point arguments.
\end{remark}
\begin{remark}
We highlight that the solvability of the Riccati equation associated with pre-committed optimal controls,
in which the time-inconsistency is ignored, remains unclear.
However, by considering its closed-loop equilibrium strategies,
we can establish the well-posedness of the corresponding ERE \eqref{ERE}.
\end{remark}

The rest of the paper is organized as follows:
In Section \ref{sec3}, we present fundamental preliminary results that underpin the subsequent analysis. Sections \ref{sec4} and \ref{sec5} derive an essential priori estimates for the equilibrium Riccati equation (ERE) \eqref{ERE}, ultimately establishing its well-posedness under the assumption of smooth coefficients. Finally, Section \ref{sec6} generalizes these findings to the non-smooth case via a mollification approximation argument, thereby concluding the proof of Theorem \ref{Th2}.

%

\section{Preliminary results}\label{sec3}

\subsection{Some useful lemmas}

In this subsection, we present several technical lemmas that will be important in our subsequent analysis.
\begin{lemma}\label{lm3.1}
Let  $\cR(\cd)\in C([0,T];\,\dbS^{n })$ be a matrix-valued function such that $\cR(t)\leq \cR(s)$ for all $0\leq t\leq s \leq T$.
Then $\cR(\cd)$ is differentiable almost everywhere and its derivative is positive semi-definite.
\end{lemma}

\begin{proof}
From the monotonicity of $\cR(\cd)$, we know that $\hbox{Trace}\big(\cR(\cd)\big)$ is a monotonically increasing function and therefore is of bounded variation.  The inequality
\begin{align*}
\big| \cR(t)-\cR(s) \big|\leq 	\big| \hbox{Trace}\big(\cR(t)\big)- \hbox{Trace}\big( \cR(s)\big) \big|,
\end{align*}
shows that each component of $\cR(\cd)$ is likewise of bounded variation. Therefore, $\cR(\cd)$ is differentiable almost everywhere on $[0,T]$.

At any point $t_0$ where $\cR(\cd)$ is differentiable, we have for arbitrary $x\in \dbR^n$ that  \vspace{-3mm}
\begin{align*}
x^{\top}\cR'(t_0) x = x^{\top}\lim_{t \to t_0 } \frac{\cR(t)-\cR(t_0)}{t-t_0} x \geq 0,
\end{align*}
which completes the proof.
\end{proof}	
\begin{lemma}{\cite[Lemma 7.3]{Yong-Zhou1999}}\label{lm3.2}
Suppose that
$\cA (\cd), \cC(\cd) \in C ([0,T];\,\dbR^{n\times n})$, $\cQ(\cd)\in C ([0,T];\,\dbS^{n})$ and $\cG \in \dbS^{n}$.
Let\vspace{-3mm}
\begin{equation*}
{\cP}(t)=\mE \Big[ \Psi(t,T)^\top \cG\Psi(t,T)+ \int_{t}^{T} \Psi(t,s)^\top  \cQ(s)\Psi(t,s)ds   \Big],  \q 0\leq t \leq T,
\end{equation*}
where\vspace{-3mm}
\begin{equation*}
\Psi(t,s)=I+\int_{t}^{s} \cA(r)\Psi(t,r)dr+\int_{t}^{s}\cC(r)\Psi(t,r)dW(r),  \q 0\leq t \leq s\leq  T.
\end{equation*}
Then we have\vspace{-3mm}
\begin{equation*}
\dot{\cP}(t)=-\cP(t)\cA(t)-\cA(t)^\top\cP(t)-\cC(t)^\top\cP(t)\cC(t)-\cQ(t),\q t\in [0,T].
\end{equation*}
\end{lemma}
\begin{lemma}\label{lm3.3}
Let $\cG(\cd)\in C([0,T];\,\dbS^m)$  and  $\cQ(\cd,\cd)\in C([0,T]^2;\, \dbS^{n})$  satisfy
\begin{equation*}
\delta I \leq\cG(t)\leq \cG(s)\leq \widehat{G}\,\text{  and  }\,  \delta I \leq\cQ(t,s)\leq \cQ(r,s)\leq \widehat{Q},\q  0\leq t \leq r \leq s \leq T,
\end{equation*}
where $\d>0$ is a fixed constant. Then there exist a sequence $\{\cG_n(\cd)\}_{n\geq 1}\subset C^{\infty}([0,T];\,\dbS^m)$ satisfying
\begin{itemize}
\item $\delta I \leq \cG_n(t)\leq \cG_n(s) \leq \widehat{G},\q  0\leq t \leq s \leq T,$
\item $   \lim_{n\to \infty}\|\cG_n(\cd)-\cG(\cd)\|_{C([0,T];\,\dbS^m)}=0,$
\end{itemize}
and a sequence $\{\cQ_n(\cd,\cd)\}_{n\geq 1}\subset C([0,T]^2;\,\dbS^{n})$ satisfying
\begin{itemize}
\item $ \text{for any }   0\leq    s \leq T,\q  \cQ_n(t,s) \text{ is infinitely differentiable with respect to $t$, }$
\item $ \delta I \leq\cQ_n(t,s)\leq \cQ_n(r,s)\leq \widehat{Q},\q  0\leq t \leq r \leq s \leq T,$
\item $\lim_{n\to \infty}\int_{0}^{T} \sup_{t\in [0,s]}
\big|  \cQ_n(t,s)  -  \cQ(t,s) \big|^2
ds=0.$
\end{itemize}
\end{lemma}
\begin{proof}
	Define
	\begin{equation*}
	\eta (t)=\begin{cases}
	\(\int_{|t|\leq 1}e^{\frac{1}{t^2-1}}dt\)^{-1} e^{\frac{1}{t^2-1}}, &|t|\leq 1,\\
	0,& |t|>1,
	\end{cases}
	\end{equation*}
	and $\eta_{\e}(t)\deq  \frac{1}{\e} \eta(\frac{t}{\e})$ for $\e>0$.
	
	\ss
	
	Let
	\begin{equation*}
	\bar{\cG}(t)=\begin{cases}
	\int_{\dbR}\cG(0)\chi_{[-2T,0]}(s)\eta_{T/4}(t-s)ds,\q & t\in (-\infty,-T],\\
	\cG(0),\q & t\in [-T,0],\\
	\cG(t),\q & t\in [0,T],\\
	\cG(T),\q & t\in [T,2T],\\
	\int_{\dbR}\cG(T)\chi_{[T,3T]}(s)\eta_{T/4}(t-s)ds,\q & t\in [2T,+\infty).
	\end{cases}
	\end{equation*}
	We know that $\bar{\cG}(\cd)$ belongs to $C_c(\dbR;\,\dbS^n)$ and is monotonically increasing over $[-T,2T]$.
	Set
	\begin{equation}\label{pr-lm2.5-eq1}
	\cG_n(t)\= \int_{\dbR} \bar{\cG}(s)\eta_{\frac{1}{n}} (t-s)ds,\q t\in \dbR.
	\end{equation}
	We know that (see \cite[pp. 713--714]{Evans-2010}, for example) $\cG_n(\cd) \in C_c^{\infty}(\dbR;\,\dbS^n)$ and
	$\lim\limits_{n\to \infty}\|\cG_n(\cd)-\cG(\cd)\|_{C([0,T];\,\dbS^n)} =0$.
	
	For any $x\in\dbR^n$ and $t\in [0,T]$,
	\begin{align*}
	x^\top \delta I x &\leq x^\top 	\cG_n(t) x =  \int_{[-1/n,1/n]}  x^\top \bar{\cG}(t-s) x\, \eta_{\frac{1}{n}} (s)ds \leq  \int_{\dbR}  x^\top \widehat{G} x\, \eta_{\frac{1}{n}} (t-s)ds= x^\top 	\widehat{G} x,
	\end{align*}
	yielding the matrix inequality\vspace{-3mm}
	\begin{equation*}
	\delta I\leq\cG_n(t) \leq \widehat{G},\q t\in [0,T].
	\end{equation*}
	Moreover, for $0 \leq t_1 \leq t_2 \leq T$ and large $n$,
	\begin{align}
	&x^\top \cG_n(t_2) x-x^\top \cG_n(t_1) x = \int_{[-1/n,1/n]} x^\top \big[\bar{\cG}(t_2-s) -\bar{\cG}(t_1-s)\big] x\,  \eta_{\frac{1}{n}} (s)ds \geq 0.\label{app-eq1}
	\end{align}
	This implies that\vspace{-3mm}
	\begin{equation*}
	\cG_n(t)\leq \cG_n(s),\q 0\leq t\leq s \leq T.
	\end{equation*}

	Next let
	\begin{align*}
	&\bar{\cQ}(t,s)\=\begin{cases}
		\int_{\dbR}\cQ(0,s)\chi_{[-2T,0]}(r)\eta_{T/4}(t-r)dr,&  (t,s)\in \{  (t,s)\,|\, 0\leq s \leq T,-\infty <t\leq  -T \},\\
		\cQ(0,s),&    (t,s)\in \{  (t,s)\,|\, 0\leq s \leq T,-T \leq t\leq 0 \},\\
	 \cQ(t,s), & (t,s)\in \{  (t,s)\,|\, 0\leq s \leq T,0\leq t \leq s \},\\
		\cQ(s,s), &   (t,s)\in \{  (t,s)\,|\, 0\leq s \leq T,s \leq t\leq 2T \},\\
		\int_{\dbR}\cQ(s,s)\chi_{[T,3T]}(r)\eta_{T/4}(t-r)dr,&  (t,s)\in \{  (t,s)\,|\, 0\leq s \leq T,2T \leq t< +\infty \}.
	\end{cases}
	\end{align*}
	Similar to the case of $\cG(\cd)$,  for any fixed $s\in [0,T]$,
	we know that $\bar{\cQ}(\cd,s)$ belongs to $C_c(\dbR;\,\dbS^n)$ and is monotonically increasing over $[-T,2T]$.
	Set
	\begin{equation*}
	\nonumber\cQ_n(t,s)\= \int_{\dbR} \bar{\cQ}(r,s)\eta_{\frac{1}{n}} (t-r)dr,\q  (t,s) \in  \dbR\times [0,T].
	\end{equation*}
	We know that, for any fixed $s\in [0,T]$,  $\cQ_n(\cd,s) \in C_c^{\infty}(\dbR;\,\dbS^n)$ and
	\begin{equation*}
	\lim_{n\to \infty}\sup_{t\in [0,s]}|\cQ_n(t,s)-\cQ(t,s)| =0,\q \forall s\in [0,T].
	\end{equation*}
	Similar to \eqref{app-eq1}, we have\vspace{-3mm}
	\begin{equation*}
	\cQ_n(t,s)\leq \cQ_n(r,s),\q 0\leq t \leq r \leq s \leq T.
	\end{equation*}
	Note for any $x\in\dbR^n$ and $0\leq t  \leq s \leq T$, it holds that
	\begin{align*}
	x^\top	\delta I  x  \leq \int_{[-1/n,1/n]}  x^\top \bar{\cQ}(t-r,s) x\, \eta_{\frac{1}{n}} (r)dr \leq  \int_{\dbR}  x^\top \widehat{Q} x\, \eta_{\frac{1}{n}} (t-s)ds = x^\top 	\widehat{Q} x.
	\end{align*}
	Therefore\vspace{-3mm}
	\begin{equation*}
	\delta I \leq\cQ_n(t,s) \leq \widehat{Q},\q 0\leq t  \leq s \leq T.
	\end{equation*}
	Consequently, by the dominated convergence theorem, we have
\begin{equation*}
	\lim_{n\to \infty}\int_{0}^{T} \sup_{t\in [0,s]}
	\big|  \cQ_n(t,s)  -  \cQ(t,s) \big|^2
	ds=0.
\end{equation*}
\end{proof}


\subsection{An auxiliary result}


In this subsection, we establish a useful auxiliary result.

For any $\Theta(\cd)\in C([0,T];\dbR^{k\times n})$, the functions $\{P_1(t,t);t\in [0,T]\}$ and $\{P_2(t);t\in [0,T]\}$ are uniquely determined through the following system of equations:
\begin{equation}\label{theta-P1P2}
\begin{cases} \ns \ds
\Phi(t,s)=I+\int_{t}^{s} A_{\Th}\Phi(t,r)dr+\int_{t}^{s}C_{\Theta}\Phi(t,r)dW(r),  &\! \!\!\!\!\!\!\!\!\!\!\!0\leq t \leq s\leq  T,\\\ns \ds
P_2(t)=H+\int_{t}^{T}\big(  P_2   A_\Theta   +  \widehat{A}_\Theta     +\widehat{C}    P_2    +\widehat{D}    P_2   C_\Theta    \big)ds,  & 0\!\leq \!t\! \leq\! T,\\ \ns \ds
{P}_1(t,t)=\mE \Big[\Phi(t,T)^\top G_1(t)\Phi(t,T)+ \int_{t}^{T} \Phi(t,s)^\top \big( Q(t,s)+\Theta   ^\top R(t,s)\Theta   \\ \ns\ds \qq\qq\qq
+ P_2   ^\top M(t,s) P_2       + C_\Theta    ^\top  {P}_2   ^\top N(t,s)  {P}_2    C_\Theta      \big)   \Phi(t,s)ds   \Big],  &0\leq t \leq T,
\end{cases}
\end{equation}
where $A_{\Theta},C_{\Theta},\widehat{A}_{\Theta}$
are defined in \eqref{Atheata} for the given  $\Theta(\cd)$. The solutions possess the following important properties:
\begin{lemma}\label{lm3.4}
Under Assumptions \ref{H1}--\ref{H2}, for any given $\Theta(\cd)\in C([0,T];\dbR^{k\times n})$,
the functions $\{P_1(t,t);t\in[0,T]\}$ and $\{P_2(t);t\in [0,T]\}$ defined by \eqref{theta-P1P2} are continuous,
and $P_1(t,t)$ is positive semi-definite for any $t\in [0,T]$. Moreover, we have
\begin{align}
&  \sup_{t\in[0,T]}\mE \Big[\sup_{s\in[t,T]}|\Phi(t,s)|^2 \Big]\leq \cC(\Theta),\label{lm3.1-eq1}\\
&   |P_2(t)-P_2(r)|\leq \cC(\Theta) |t-r|,\q \forall 0\leq t\leq r \leq T,\label{lm3.1-eq2}
\end{align}
and
\begin{align}
|P_1(t,t)-P_1(r,r)|&\leq\cC(\Th)\(|t-r|^{1\over 2}+|G_1(t)-G_1(r)|\)\nonumber\\
&   \q+\cC(\Th)\int_r^T \Big(|Q(t,s)-Q(r,s)|+|R(t,s)-R(r,s)|\label{lm3.1-eq3} \\
& \qq\qq\qq + |M(t,s)-M(r,s)|+|N(t,s)-N(r,s)|	\Big)ds, \q\forall 0\leq t\leq r \leq T,\nonumber
\end{align}
where $\cC(\Theta)$ is a constant only depending  on $\| \Theta(\cd) \|_{C([0,T];\dbR^{k\times n})}$ and system parameters.
\end{lemma}
\begin{proof}
From Assumption \ref{H2}, the positive semi-definiteness of $P_1(t,t)$ follows immediately. The continuity of $P_2(\cd)$ is established through the Lipschitz estimate:
\begin{equation}\label{est-P2}
|P_2(t)-P_2(r)|\leq \cC(\Theta) |t-r|,\q \forall 0\leq t\leq r \leq T,
\end{equation}
where $\cC(\Theta)$ depends on $\| \Theta(\cd) \|_{C([0,T];\dbR^{k\times n})}$ and system parameters.

Next we prove that $\{P_1(t,t);t \in [0,T]\}$ is continuous. We proceed as follows. First, standard SDE estimates (e.g.,\cite[Section 3.1]{Lu-Zhang-2021}) yield
\begin{equation}\label{est-The}
\sup_{t\in[0,T]}\mE \Big[\sup_{s\in[t,T]}|\Phi(t,s)|^2 \Big]\leq \cC(\Theta).
\end{equation}
Secondly, for $0\leq t\leq r\leq s \leq T$, we have
\begin{align*}
&	\mE \big[ | \Phi(t,s)-\Phi(r,s)|^2 \big]\\
&\leq\!  2  \mE \[ \big| \Phi(t,r)\!-\!\Phi(r,r)\big|^2\!\!
+\!\Big|\!\int_r^{s}\! A_{\Theta}(\t)\big( \Phi(t,\t)\!-\!\Phi(r,\t) \big)d\t \Big|^2\!\! +\Big|\!\int_r^{s}\! C_{\Theta}(\t)\big( \Phi(t,\t)\!-\!\Phi(r,\tau) \big)dW(\t) \Big|^2\]\\
&\leq2\mE\[2\Big|\int_{t}^rA_{\Theta}(\t)\Phi(t,\t)d\t\Big|^2
+2\Big|\int_{t}^r  C_{\Theta}(\t)\Phi(t,\t)dW(\t) \Big|^2 +\Big|\int_r^{s} A_{\Theta}(\t)\big( \Phi(t,\t)-\Phi(r,\t) \big)d\t \Big|^2 \\
&\qq +\Big|\int_r^{s} C_{\Theta}(\t)\big( \Phi(t,\t)-\Phi(r,\tau) \big)dW(\t) \Big|^2\]\\
&\leq \cC(\Th) (r-t)+ \cC(\Th) \int_r^{s} \mE\big[| \Phi(t,\t)-\Phi(r,\t)|^2 \big]d\t.
\end{align*}
Gronwall's inequality then gives
\begin{equation}\label{est-Phi}
\sup_{r \leq s \leq T}\mE \big[ | \Phi(t,s)-\Phi(r,s) |^2\big]\leq \cC(\Th)|t-r|.
\end{equation}
We  now  estimate $|P_1(t,t)-P_1(r,r)|$ for   $0\leq t\leq r \leq T$:
\begin{align*}
&|P_1(t,t)-P_1(r,r)|\\&\leq \cC(\Th) \mE \big[ | \Phi(t,T)-\Phi(r,T) |^2\big]^{1\over 2}
\sup_{s\in[0,T]}\mE\big[|\Phi(s,T)|^2\big]^{1\over 2}\\
&\q+\cC(\Th)|G_1(t)-G_1(r)|\sup_{s\in[0,T]}\mE\big[|\Phi(s,T)|^2\big] +\cC(\Th)\int_t^r \dbE\big[| \Phi(t,s) |^2\big]ds\\
&\q+\cC(\Th)\int_r^T \dbE\big[| \Phi(t,s)-\Phi(r,s) |^2\big]^{1\over 2}ds\sup_{t\in[0,T]}\mE\[\sup_{s\in[t,T]}|\Phi(t,s)|^2\]^{1\over 2}\\
&\q+\cC(\Th)\int_r^T \Big(|Q(t,s)-Q(r,s)|+|R(t,s)-R(r,s)|\\
&\qq \qq \qq \q  + |M(t,s)-M(r,s)|+|N(t,s)-N(r,s)| \Big)ds
\sup_{t\in[0,T]}\mE\[\sup_{s\in[t,T]}|\Phi(t,s)|^2\].
\end{align*}
This, together with \eqref{est-The}--\eqref{est-Phi}, implies that
\begin{align*}
&	|P_1(t,t)-P_1(r,r)|\\ &\leq\cC(\Th)\Big(|t-r|^{1\over 2}+|G_1(t)-G_1(r)|\Big)\nonumber\\
&\q+\cC(\Th)\int_r^T \Big(|Q(t,s)-Q(r,s)|+|R(t,s)-R(r,s)| + |M(t,s)-M(r,s)|+|N(t,s)-N(r,s)|	\Big)ds.
\end{align*}
This completes the proof.
\end{proof}
\begin{remark}\rm \label{rm3.1}
It is clear that the above result still holds for  $\Theta(\cd)\in L^\infty(0,T;\dbR^{k\times n})$.
\end{remark}

\subsection{Integral equation systems and equilibrium value function}

In this subsection, we demonstrate that the solvability of ERE \eqref{ERE} is equivalent to that of an integral equation system (IES, for short).
Based on this equivalence, we derive the equilibrium value function for Problem (TI-FBSLQ).

We first introduce the following IES:
\begin{equation}\label{Inte}
\begin{cases} \ns \ds
\Phi(t,s)=I+\int_{t}^{s} A_{\Th}\Phi(t,r)dr+\int_{t}^{s}C_{\Theta}\Phi(t,r)dW(r),  &\! \!\!\!\!\!\!\!\!\!\!\!0\leq t \leq s\leq  T,\\\ns \ds
P_2(t)=H+\int_{t}^{T}\big(  P_2   A_\Theta   +  \widehat{A}_\Theta     +\widehat{C}    P_2    +\widehat{D}    P_2   C_\Theta    \big)ds,  & 0\!\leq \!t\! \leq\! T,\\ \ns \ds
{P}_1(t,t)=\mE \Big[\Phi(t,T)^\top G_1(t)\Phi(t,T)+ \int_{t}^{T} \Phi(t,s)^\top \big( Q(t,s)+\Theta   ^\top R(t,s)\Theta   \\ \ns\ds \qq\qq\qq
+ P_2   ^\top M(t,s) P_2       + C_\Theta    ^\top  {P}_2   ^\top N(t,s)  {P}_2    C_\Theta      \big)   \Phi(t,s)ds   \Big],  &0\leq t \leq T,
\end{cases}
\end{equation}
where
\begin{align}\label{Inte-theta}
\nonumber	\Theta(s)=& -\big[R(s,s) +  D(s)^\top \!\big( {P}_1(s,s) +  P_2(s)^\top N(s,s) P_2(s) \big)D(s)\big]^{-1} \\  \ns \ds &
\times  \big[ B(s)^\top {P}_1(s,s)     +  D(s)^\top\big({P}_1(s,s)  +  P_2(s)^\top N(s,s) P_2(s) \big)C(s)   \nonumber\\  \ns \ds  &\q
+ \big(\widehat{B}(s)^\top\! +  {B}(s)^\top  P_{2}(s)^\top   +  D(s)^\top  P_{2}(s)^\top   \widehat{D}(s)^\top\big ) G_2(s) P_{2}(s) \big].
\end{align}

We have the following result.

\begin{lemma}\label{lm3.5}
Under Assumptions \ref{H1}--\ref{H2}, the following statements are equivalent:

(i) The  ERE  \eqref{ERE} admits a unique solution $(P_1(\cd,\cd),P_2(\cd))\in C(\Delta^*[0,T];\dbS^n_+)\times C([0,T];\dbR^{m \times n})$;

(ii) The  IES  \eqref{Inte} admits a unique solution $(\{P_1(t,t);t\in [0,T]\},P_2(\cd))\in  C([0,T];\dbS^n_+) \times C([0,T];\dbR^{m \times n})$.

Furthermore, both systems share identical  functions $\Theta(\cd)$.
\end{lemma}

\begin{proof}
(i)$\Rightarrow$(ii). From the  positive semi-definite solution $P_1(\cd,\cd)$  and \ref{H2}, we have
$
R(s,s)+D(s)^\top P_1(s,s)D(s)\\\geq \d I,
$
for $ s\in[0,T]$. Consequently, the function $\Th(\cd)$ defined in \eqref{Theta} remains uniformly bounded.  Substituting this into the first equation of \eqref{Inte}, we obtain for each fixed $t\in[0,T)$ a unique solution $\Phi(t,\cd)\in L_\dbF^2(\Om;\,C([t,T];\,\dbR^{n \times n   }))$.
Applying  It\^{o}'s formula to the matrix-valued process $s\mapsto\Phi(t,s)^\top  P_1(t,s)\Phi(t,s)$ over $[t,T]$ and taking expectations yields
\begin{align*}
&{P}_1(t,t)- \mE \Phi(t,T)^\top G(t)\Phi(t,T)\\
& =\mE \Big[\int_{t}^{T}\Phi(t,s)^\top \big( Q(t,s)+\Theta(s)^\top R(t,s)\Theta(s) +P_2(s)^\top M(t,s)P_2(s)\\
&\qq \qq+ C_{\Theta}(s)^\top P_2(s)^\top N(t,s) P_2(s)C_{\Theta}(s) \big)\Phi(t,s)ds\Big].
\end{align*}
This establishes that the diagonal values $\{P_1(t,t);t\in [0,T]\}$ and $P_2(\cd)$  indeed solve the IES \eqref{Inte}.

(ii)$\Rightarrow$(i). Let the coefficients $A_\Th(\cd)$, $C_\Th(\cd)$, and $\Th(\cd)$ be defined by \eqref{Atheata} and \eqref{Inte-theta} respectively. Consider the coupled system:
\begin{equation}\label{lm3.4-pr-eq1}
\begin{cases}\ds
\frac{\wt P_1(t,s)}{ds} + \wt P_{1}(t,s) A_\Theta  +A_\Theta^\top  \wt P_1(t,s) +C_\Theta^\top \wt P_{1}(t,s)C_\Theta  +Q(t,s) \\\ns\ds\q
+\Th^\top R(t,s)\Th + \wt P_2^\top M(t,s)\wt P_2  + C_\Theta ^\top \wt P_2^\top N(t,s) \wt P_2 C_\Theta=0, & 0\leq t\leq s\leq T,\\\ns\ds
\frac{d\wt P_2}{ds}+ \wt P_2A_\Theta+  \widehat{A}_\Theta  +\widehat{C} \wt P_2 +\widehat{D} \wt P_2C_\Theta =0, & 0\leq t\leq s\leq T,\\\ns\ds
\wt P_1(t,T)=G_1(t),	\q \wt P_2(T)=H, & 0\leq t\leq T,
\end{cases}
\end{equation}
and the associated SDE
\begin{equation}\label{lm3.4-pr-eq2}
\wt \Phi(t,s)=I+\int_{t}^{s} A_{\Th}(r)\wt \Phi(t,r)dr+\int_{t}^{s}C_{\Theta}(r)\wt \Phi(t,r)dW(r),  \q 0\leq t \leq s\leq  T.
\end{equation}
By standard  results for ODEs and SDEs, the systems \eqref{lm3.4-pr-eq1} and \eqref{lm3.4-pr-eq2} admit unique solutions $(\wt P_1(\cd,\cd), \wt P_2(\cd))$, and $\wt\Phi(\cd,\cd)$, respectively.

Since $\Phi(\cd,\cd)$ and $\wt \Phi(\cd,\cd)$ solve the same linear SDE, and $ P_2(\cd)$ and $\wt P_2(\cd)$ solve the same ODE,
we have $\Phi(\cd,\cd)=\wt \Phi(\cd,\cd)$ and $ P_2(\cd)=\wt P_2(\cd)$.
Applying  It\^{o}'s formula to $s\mapsto\wt\Phi(t,s)^\top \wt P_1(t,s) \wt \Phi(t,s)$ over $[t,T]$ gives
\begin{align}
\wt {P}_1(t,t)&= \mE \[ \wt \Phi(t,T)^\top G(t) \wt \Phi(t,T)+\int_{t}^{T} \wt \Phi(t,s)^\top \big(  Q(t,s)+\Theta(s)^\top R(t,s)\Theta(s) \nonumber   \\
&\qq +\wt P_2(s)^\top M(t,s)\wt P_2(s)  + C_{\Theta}(s)^\top \wt P_2(s)^\top N(t,s) \wt P_2(s)C_{\Theta}(s) \big) \wt \Phi(t,s)ds \]\nonumber\\
& =P_1(t,t),\q t\in[0,T].\label{lm3.4-pr-eq3}
\end{align}
Consequently,  the function $\Theta(\cd)$ in  \eqref{Inte-theta}  admits the representation
\begin{align}
\Theta(s)=& -\big[R(s,s) +  D(s)^\top \!\big( \wt {P}_1(s,s) +  \wt P_2(s)^\top N(s,s)  \wt P_2(s) \big)D(s)\big]^{-1} \nonumber\\  \ns \ds &
\times \big[ B(s)^\top  \wt {P}_1(s,s)     +  D(s)^\top\big( \wt {P}_1(s,s)  +   \wt P_2(s)^\top N(s,s) \wt P_2(s) \big)C(s)   \nonumber\\  \ns \ds  &\q
+ \big(\widehat{B}(s)^\top +  {B}(s)^\top  \wt P_{2}(s)^\top   +  D(s)^\top  \wt P_{2}(s)^\top   \widehat{D}(s)^\top\big ) G_2(s) \wt P_{2}(s) \big].  \label{lm3.4-pr-eq4}
\end{align}
From \eqref{lm3.4-pr-eq1}--\eqref{lm3.4-pr-eq4}, we can see that $ (\wt P_1(\cd,\cd),\wt P_2(\cd))$ is a solution to the ERE \eqref{ERE}.

Moreover, from the proof above, we can get that ERE \eqref{ERE} and IES \eqref{Inte} share the same $\Theta(\cd)$.
Since $\Theta(\cd)$ uniquely determines the solutions of ERE \eqref{ERE} and IES \eqref{Inte},
we conclude that the uniqueness of the solution of ERE \eqref{ERE}  is equivalent to that of IES \eqref{Inte}.
\end{proof}

Next, we employ Lemma \ref{lm3.5} to derive the equilibrium value function for Problem (TI-FBSLQ).

\begin{lemma}\label{lm3.6}
Under Assumptions \ref{H1}--\ref{H2}, suppose the  ERE \eqref{ERE}  admits a solution $(P_1(\cd,\cd),P_2(\cd))\in C(\Delta^*[0,T];\dbS^n_+)\times C([0,T];\dbR^{m \times n})$.
Then the equilibrium value function $\dbV(\cd,\cd)$ admits the explicit quadratic form:
\begin{equation}
\dbV(t,x)= \frac{1}{2}\big\langle \big(P_1(t,t)+ P_2(t)^\top G_2(t) P_2(t) \big) x ,x \big\rangle,\q \forall (t,x)\in [0,T]\times L_{\cF_t}^2(\Omega;\dbR^n).
\end{equation}
\end{lemma}

\begin{proof}
Substituting the closed-loop equilibrium strategy $\Theta(\cd)$ from \eqref{Theta} into the state equation \eqref{state}, we denote the corresponding solution by $(\bar{X}(\cd),\bar{Y}(\cd))$. Direct computation yields
\begin{align*}
\bar{Y}(\cd)= P_2(\cd) \bar{X}(\cd),\q \bar{Z}(\cd)= P_2(\cd) C_{\Theta}(\cd) \bar{X}(\cd),
\end{align*}
where $P_2(\cd)$ solves \eqref{ERE}. Noting that $\bar{X}(\cd)=\Phi(t,\cd)x$ with $\Phi(t,\cd)$ solving \eqref{Inte}, Lemma \ref{lm3.5} gives
\begin{eqnarray*}
\dbV(t,x)&\3n=&\3n\cJ(t,x;\Theta(\cd)\bar{X}(\cd))\\
&\3n=&\3n\frac{1}{2}\mE_t\Big[ \int_{t}^{T} \big( \big \langle Q(t,s)\bar{X},\bar{X}  \big\rangle +  \big\langle R(t,s)\Theta\bar{X},\Theta\bar{X} \big \rangle + \langle M(t,s)P_2 \bar{X},P_2 \bar{X} \rangle \\
&& \qq \q +\langle N(t,s)P_2 C_{\Theta} \bar{X},P_2 C_{\Theta} \bar{X} \rangle \big)ds+  \langle G_1(t) \bar{X}(T) ,\bar{X}(T) \rangle + \langle \bar{Y}(t),\bar{Y}(t)  \rangle \Big]\\
&\3n=&\3n \frac{1}{2} \mE_t \Big[ \big\langle \Phi(t,T)^\top G_1(t)\Phi(t,T)x,x \big \rangle + \int_{t}^{T} \big\langle\Phi(t,s)^\top \big( Q(t,s)+\Theta   ^\top R(t,s)\Theta   \\ \ns\ds \qq\qq\qq
&& \qq \q + P_2   ^\top M(t,s) P_2       + C_\Theta    ^\top  {P}_2   ^\top N(t,s)  {P}_2    C_\Theta      \big)   \Phi(t,s) x,x \big\rangle ds   +  \big\langle  P_2(t)^\top G_2(t)P_2(t)x,x \big\rangle   \Big] 		\\
&\3n=&\3n \frac{1}{2}\big\langle \big( P_1(t,t)+ P_2(t)^\top G_2(t) P_2(t) \big) x ,x \big\rangle.
\end{eqnarray*}
This completes the proof.
\end{proof}

\section{Priori estimates for the ERE \eqref{ERE} with smooth coefficients}\label{sec4}

In this section, we derive a priori estimates for solution to the  ERE  \eqref{ERE} under smooth coefficient assumptions. The following technical conditions will be imposed:
\begin{taggedassumption}{(H4)}\label{H4}
The coefficients  satisfy:
\begin{align}
&A(\cd),C(\cd) \in C([0,T];\dbR^{n\times n}),\q B(\cd),D(\cd) \in C([0,T];\dbR^{n\times k}), \nonumber\\&
\widehat{C}(\cd),\widehat{D}(\cd)\in C([0,T];\dbR^{m\times m}),\q \widehat{A}(\cd)\in C([0,T];\dbR^{m\times n}),\q \widehat{B}(\cd)\in C([0,T];\dbR^{m\times k}). \label{H4-coe}
\end{align}
\end{taggedassumption}
\begin{taggedassumption}{(H5)}\label{H5}
The weighting matrices satisfy:
\begin{align}
&Q(\cd,\cd)\in C([0,T]^2;\dbS^n),\q R(\cd,\cd)\in C([0,T]^2;\dbS^k),\q M(\cd,\cd),N(\cd,\cd)\in C([0,T]^2;\dbS^m),\nonumber \\&
G_1(\cd)\in C^1([0,T];\dbS^n),\q G_2(\cd)\in C^1([0,T];\dbS^m), \text{ and    }\q  Q(t,s),R(t,s),N(t,s),M(t,s) \nonumber\\&
\text{ are differentiable for t over } \{(t,s)\,|\, 0\leq t \leq s \leq T \}. \label{H5-coe}
\end{align}
\end{taggedassumption}
\begin{proposition}\label{pri-est-Pi}
Under Assumptions \ref{H2}--\ref{H5}, suppose the  ERE \eqref{ERE}  admits a solution $(P_1(\cd,\cd),P_2(\cd))\in C(\Delta^*[0,T];\dbS^n_+)\times C([0,T];\dbR^{m \times n})$.
Then we have
\begin{equation}
\sup_{t\in[0,T]}\big|P_1(t,t)+P_2(t)^\top G_2(t)P_2(t) \big|\leq \cC^*_{\dbV},
\end{equation}
where $\cC^*_{\dbV}$ is defined in \eqref{def-CV}.
\end{proposition}

\begin{proof}
Following Lemma \ref{lm3.5}, we focus on establishing the required estimates for the  IES \eqref{Inte}. Let us define the key auxiliary function: \vspace{-3mm}
$$\dbV(t)\=P_1(t,t)+P_2(t)^\top G_2(t)P_2(t).$$
From Lemma \ref{lm3.6}, we recognize that $ \frac{1}{2} \langle \dbV(t)x,x \rangle$ represents the equilibrium value function for Problem (TI-FBSLQ). Under Assumption \ref{H2}, we obtain the following component bounds:\vspace{-2mm}
\begin{align*}
0\leq P_1(t,t)\leq \dbV(t),\q
0\leq \delta P_2(t)^\top P_2(t)\leq \dbV(t),
\end{align*}
which implies\vspace{-2mm}
\begin{equation}\label{pr-eq9}
|P_1(t,t)|\leq |\dbV(t)|,\q |P_2(t)|^2\leq |P_2(t)^\top P_2(t)|_{\rm Tr}\leq \frac{1}{\delta} |\dbV(t)|_{\rm Tr} \leq  \frac{n}{\delta} |\dbV(t)|.
\end{equation}	

Next we take the derivative of $\dbV(t)$.
Firstly, for any fixed  $t,t_0\in [0,T]$, from \eqref{Inte} we have the following decomposition
\begin{equation*}
\frac{P_1(t,t)-P_1(t_0,t_0)}{t-t_0}= \dbI_1+ \dbI_2,
\end{equation*}
where\vspace{-2mm}
{\small
\begin{align*}
\dbI_1&= \frac{1}{t-t_0}\bigg \{\mE \Big[ \Phi(t,T)^\top G_1(t)\Phi(t,T)+ \int_{t}^{T} \Phi(t,s)^\top \big( Q(t,s)+\Theta(s)^\top R(t,s)\Theta(s)\\&\qq \qq
+P_2(s)^\top M(t,s) P_2(s)+ C_{\Theta}(s)^\top P_2(s)^\top N(t,s) P_2(s) C_{\Theta}(s) \big)\Phi(t,s)ds\Big]\\&\qq\qq
-\mE\[\Phi(t,T)^\top G_1(t_0)\Phi(t,T)+\int_{t}^{T} \Phi(t,s)^\top \big(Q(t_0,s)+\Theta(s)^\top R(t_0,s)\Theta(s)\\&\qq\qq\q
+P_2(s)^\top M(t_0,s) P_2(s)+ C_{\Theta}(s)^\top P_2(s)^\top N(t_0,s) P_2(s) C_{\Theta}(s) \big)\Phi(t,s)ds\Big] \bigg\},\\
\dbI_2&=\frac{1}{t-t_0}\bigg \{\mE\[ \Phi(t,T)^\top G_1(t_0)\Phi(t,T)+ \int_{t}^{T} \Phi(t,s)^\top \big( Q(t_0,s)+\Theta(s)^\top R(t_0,s)\Theta(s)\\&\qq\qq
+P_2(s)^\top M(t_0,s) P_2(s)+ C_{\Theta}(s)^\top P_2(s)^\top N(t_0,s) P_2(s) C_{\Theta}(s)\big)\Phi(t,s)ds  \Big]\\&\qq\qq
-\mE\[\Phi(t_0,T)^\top G_1(t_0)\Phi(t_0,T)+ \int_{t_0}^{T} \Phi(t_0,s)^\top \big(Q(t_0,s)+\Theta(s)^\top R(t_0,s)\Theta(s)\\&\qq\qq\q
+P_2(s)^\top M(t_0,s) P_2(s)+ C_{\Theta}(s)^\top P_2(s)^\top N(t_0,s) P_2(s) C_{\Theta}(s) \big)\Phi(t_0,s)ds \] \bigg\}.	
\end{align*}}

For $\dbI_1$, applying the dominated convergence theorem and Lemma \ref{lm3.1} yields\vspace{-2mm}
\begin{align}\label{pr-eq5}
\lim_{t \to t_0} \dbI_1 &= \mE \Big[\Phi(t_0,T)^\top G'(t_0)\Phi(t_0,T)+ \int_{t_0}^{T} \Phi(t_0,s)^\top\big(\partial_t Q(t_0,s)+\Theta(s)^\top\partial_t R(t_0,s)\Theta(s)\nonumber \\&\qq
+P_2(s)^\top \partial_t M(t_0,s) P_2(s)+ C_{\Theta}(s)^\top P_2(s)^\top \partial_t N(t_0,s) P_2(s) C_{\Theta}(s)  \big)\Phi(t_0,s)ds\Big] \geq 0,
\end{align}
for any $ t_0 \in [0,T]$.

We next take the limit for $\dbI_2$.
For any fixed $t_0 \in [0,T]$, we introduce the following auxiliary function over $[0,T]$:\vspace{-2mm}
\begin{align} \label{de-cP}
\cP (t)&\deq\mE\[ \Phi(t,T)^\top G(t_0)\Phi(t,T)+ \int_{t}^{T} \Phi(t,s)^\top \big(Q(t_0,s)+\Theta(s)^\top R(t_0,s)\Theta(s)\nonumber\\&\qq
+P_2(s)^\top M(t_0,s) P_2(s)+ C_{\Theta}(s)^\top P_2(s)^\top N(t_0,s) P_2(s) C_{\Theta}(s) \big)\Phi(t,s)ds\],
\end{align}
where $\Phi(\cd,\cd)$ is uniquely determined by \eqref{Inte}. From \eqref{de-cP}, we see that
\begin{equation}\label{pr-eq1}
\cP(t_0)= P_1(t_0,t_0)\q\hbox{and}\q \dbI_2=\frac{\cP(t)-\cP(t_0)}{t-t_0}.
\end{equation}
Moreover, by Lemma \ref{lm3.2},  we have
\begin{align}\label{pr-eq2}
\dot{\cP}(t) &=-\cP(t)A_{\Theta}(t)-A_{\Theta}(t)^\top\cP(t)-C_{\Theta}(t)^\top\cP(t)C_{\Theta}(t)- Q(t_0,t)-\Theta(t)^\top R(t_0,t)\Theta(t)\nonumber\\&\q
-P_2(t)^\top M(t_0,t) P_2(t)- C_{\Theta}(t)^\top P_2(t)^\top N(t_0,t) P_2(t) C_{\Theta}(t)
\end{align}
with\vspace{-2mm}
\begin{align}\label{pr-eq3}
\nonumber	\Theta(t)\!=&\! -\!\big[R(t,t) \!+ \! D(t)^\top \!\big( {P}_1(t,t)\! + \!  P_2(t)^\top N(t,t) P_2(t) \big)D(t)\big]^{-1} \\  \ns \ds &
\times  \big[ B(t)^\top {P}_1(t,t)     +  D(t)^\top\big({P}_1(t,t)  +  P_2(t)^\top N(t,t) P_2(t) \big)C(t)   \nonumber\\  \ns \ds  &\q
+ \big(\widehat{B}(t)^\top\! +  {B}(t)^\top  P_{2}(t)^\top   +  D(t)^\top  P_{2}(t)^\top   \widehat{D}(t)^\top\big ) G_2(t) P_{2}(t) \big].
\end{align}
From \eqref{pr-eq1}--\eqref{pr-eq3}, we know that
\begin{align}\label{pr-eq6}
\dot{\cP}(t_0)
&=-\cP(t_0)A_{\Theta}(t_0)-A_{\Theta}(t_0)^\top\cP(t_0)-C_{\Theta}(t_0)^\top\cP(t_0)C_{\Theta}(t_0)-Q(t_0,t_0)-\Theta(t_0)^\top R(t_0,t_0)\Theta(t_0)\nonumber\\
&\q -P_2(t_0)^\top M(t_0,t_0) P_2(t_0) - C_{\Theta}(t_0)^\top P_2(t_0)^\top N(t_0,t_0) P_2(t_0) C_{\Theta}(t_0)\nonumber\\
&=-P_1(t_0,t_0)A_{\Theta}(t_0)\!-\!A_{\Theta}(t_0)^\top P_1(t_0,t_0)\!-\!C_{\Theta}(t_0)^\top P_1(t_0,t_0)C_{\Theta}(t_0)\!-\!Q(t_0,t_0)\!-\!\Theta(t_0)^\top R(t_0,t_0)\Theta(t_0)\nonumber \\
&\q -P_2(t_0)^\top M(t_0,t_0) P_2(t_0) - C_{\Theta}(t_0)^\top P_2(t_0)^\top N(t_0,t_0) P_2(t_0) C_{\Theta}(t_0),\q   t_0 \in [0,T].
\end{align}
Thus for $\dbI_2$, we have\vspace{-2mm}
\begin{align}\label{pr-eq7}
\lim_{t \to t_0} \dbI_2 = \lim_{t \to t_0} \frac{\cP(t)-\cP(t_0)}{t-t_0} =\dot{\cP}(t_0), \q t_0 \in [0,T].
\end{align}
Secondly,  \vspace{-2mm}
\begin{align}\label{pr-eq4}
&\frac{d\(P_2(t)^\top G_2(t)P_2(t)\)}{dt}\nonumber\\
&= \dot{P_2}(t)^\top G_2(t)P_2(t)+ P_2(t)^\top G_2(t)\dot{P_2}(t) +P_2(t)^\top G'_2(t)P_2(t)\nonumber\\
&\geq -\big(P_2(t)A_\Theta(t)+  \widehat{A}_\Theta(t)  +\widehat{C}(t)P_2(t) +\widehat{D}(t) P_2(t)C_\Theta(t)\big)^\top G_2(t) P_2(t)\nonumber\\
&\q -P_2(t)^\top G_2(t) \big( P_2(t)A_\Theta(t)+  \widehat{A}_\Theta(t)  +\widehat{C}(t)P_2(t) +\widehat{D}(t) P_2(t)C_\Theta(t) \big),
\end{align}
where the last inequality is due to $G'_2(t)\geq 0$.

Combining \eqref{pr-eq5} and \eqref{pr-eq6}--\eqref{pr-eq4},
consequently,  for any $ t_0 \in [0,T]$, we obtain
\begin{align}\label{pr-eq8}
\dot{\dbV}(t_0)&= \lim_{t \to t_0} \frac{P_1(t,t)-P_1(t_0,t_0)}{t-t_0} +\big(P_2(t)^\top G_2(t)P_2(t)\big)'\big|_{t=t_0}\nonumber\\
& =\lim_{t \to t_0} \dbI_1+ 	\lim_{t \to t_0} \dbI_2+\big(P_2(t)^\top G_2(t)P_2(t)\big)'\big|_{t=t_0}\nonumber\\
& \geq -\big( P_1(t_0,t_0)A_{\Theta}(t_0)+A_{\Theta}(t_0)^\top P_1(t_0,t_0)+C_{\Theta}(t_0)^\top P_1(t_0,t_0)C_{\Theta}(t_0) \nonumber\\
&\q+Q(t_0,t_0)+\Theta(t_0)^\top R(t_0,t_0)\Theta(t_0) +P_2(t_0)^\top M(t_0,t_0) P_2(t_0)\nonumber\\
&\q+ C_{\Theta}(t_0)^\top P_2(t_0)^\top N(t_0,t_0) P_2(t_0) C_{\Theta}(t_0)   \big)	  \nonumber\\
&\q -P_2(t)^\top G_2(t) \big( P_2(t)A_\Theta(t)+  \widehat{A}_\Theta(t)  +\widehat{C}(t)P_2(t) +\widehat{D}(t) P_2(t)C_\Theta(t) \big)\nonumber\\
&\q 	 -\big( P_2(t)A_\Theta(t)+  \widehat{A}_\Theta(t)  +\widehat{C}(t)P_2(t) +\widehat{D}(t) P_2(t)C_\Theta(t)\big)^\top G_2(t) P_2(t).
\end{align}
Recall  \eqref{pr-eq3} and the definition of $\dbV(\cd)$, we know that
\begin{align*}
&-\!\big[R(t,t) \!+ \! D(t)^\top \!\big( {P}_1(t,t)\! + \!  P_2(t)^\top N(t,t) P_2(t) \big)D(t)\big]^{-1}\Theta(t)\\  \ns \ds
&= \big[ B(t)^\top \dbV(t)     +  D(t)^\top\big({P}_1(t,t)  +  P_2(t)^\top N(t,t) P_2(t) \big)C(t)   \nonumber\\  \ns \ds
&\q + \big(\widehat{B}(t)^\top\!   +  D(t)^\top  P_{2}(t)^\top   \widehat{D}(t)^\top\big ) G_2(t) P_{2}(t) \big],
\end{align*}
which implies\vspace{-4mm}
\begin{align*}
&-2\Theta(t)^\top\!\big[R(t,t) \!+ \! D(t)^\top \!\big( {P}_1(t,t)\! + \!  P_2(t)^\top N(t,t) P_2(t) \big)D(t)\big]^{-1}\Theta(t)\\  \ns \ds
&= \Theta(t)^\top \big[ B(t)^\top \dbV(t)     +  D(t)^\top\big({P}_1(t,t)  +  P_2(t)^\top N(t,t) P_2(t) \big)C(t)   \nonumber\\  \ns \ds
&\q + \big(\widehat{B}(t)^\top\!   +  D(t)^\top  P_{2}(t)^\top   \widehat{D}(t)^\top\big ) G_2(t) P_{2}(t) \big]   \\
&\q +\big[ \dbV(t) B(t) + C(t)^\top \big( P_1(t,t)+P_2(t)^\top N(t,t)P_2(t) \big)D(t) \\
&\qq + P_2(t)^\top G_2(t) \big( \widehat{B}(t)+\widehat{D}(t)P_2(t)D(t) \big)   \big] \Theta(t).
\end{align*}
Plugging this into \eqref{pr-eq8}, via direct computation,  we get
\begin{align*}
\dot{\dbV}(t_0)&\geq -\big( P_1(t_0,t_0)+ P_2(t_0)^\top G_2(t_0) P_2(t_0) \big) A(t_0)- A(t_0)^\top \big( P_1(t_0,t_0)+ P_2(t_0)^\top G_2(t_0) P_2(t_0) \big) \\
&\q - C(t_0)^\top P_1(t_0,t_0)C(t_0) -Q(t_0,t_0)-P_2(t_0)^\top M(t_0,t_0) P_2(t_0)\\
&\q - C(t_0)^\top P_2(t_0)^\top N(t_0,t_0) P_2(t_0) C(t_0) \\
&\q -P_2(t_0)^\top G_2(t_0) \big(  \widehat{A}(t_0)  +\widehat{C}(t_0)P_2(t_0) +\widehat{D}(t_0) P_2(t_0)C(t_0) \big)\\
&\q -\big(  \widehat{A}(t_0)  +\widehat{C}(t_0)P_2(t_0) +\widehat{D}(t_0) P_2(t_0)C(t_0)\big)^\top G_2(t_0) P_2(t_0)-\Theta(t_0)^\top\\
&\q \times \Big( B(t_0)^\top \big( P_1(t_0,t_0)+ P_2(t_0)^\top G_2(t_0) P_2(t_0) \big)     \\
&\qq   +  D(t_0)^\top\big({P}_1(t_0,t_0)  +  P_2(t_0)^\top N(t_0,t_0) P_2(t_0) \big)C(t_0)   \nonumber\\  \ns \ds
&\qq   + \big(\widehat{B}(t_0)^\top\!   +  D(t_0)^\top  P_{2}(t_0)^\top   \widehat{D}(t_0)^\top\big ) G_2(t_0) P_{2}(t_0) \Big)   \\
&\q -\Big( \big( P_1(t_0,t_0)+ P_2(t_0)^\top G_2(t_0) P_2(t_0) \big)  B(t_0) \\
&\qq + C(t_0)^\top \big( P_1(t_0,t_0)+P_2(t_0)^\top N(t_0,t_0)P_2(t_0) \big)D(t_0) \\
&\qq + P_2(t_0)^\top G_2(t_0) \big( \widehat{B}(t_0)+\widehat{D}(t_0)P_2(t_0)D(t_0) \big)   \Big) \Theta(t_0)\\
&\q -\Theta(t_0)^\top\!\big[R(t_0,t_0) \!+ \! D(t_0)^\top \!\big( {P}_1(t_0,t_0)\! + \!  P_2(t_0)^\top N(t_0,t_0) P_2(t_0) \big)D(t_0)\big]^{-1}\Theta(t_0)\\
&= -\Big( \dbV(t_0)A(t_0)+A(t_0)^\top\dbV(t_0)+C(t_0)^\top P_1(t_0,t_0)C(t_0) \\
&\qq+Q(t_0,t_0)+P_2(t_0)^\top M(t_0,t_0) P_2(t_0)+ C(t_0)^\top P_2(t_0)^\top N(t_0,t_0) P_2(t_0) C(t_0)  \\
& \qq  - \Theta(t_0)^\top \big( R(t_0,t_0)+ D(t_0)^\top P_1(t_0,t_0)D(t_0) +D(t_0)^\top P_2(t_0)^\top N(t_0,t_0) P_2(t_0)D(t_0) \big)  \Theta(t_0)  \Big)\\
&\q -P_2(t_0)^\top G_2(t_0) \big(   \widehat{A}(t_0)  +\widehat{C}(t_0)P_2(t_0) +\widehat{D}(t_0) P_2(t_0)C(t_0) \big)\\
&\q 	 -\big(  \widehat{A}(t_0)  +\widehat{C}(t_0)P_2(t_0) +\widehat{D}(t_0) P_2(t_0)C(t_0)\big)^\top G_2(t_0) P_2(t_0)\\
&\geq   -\Big( \dbV(t_0)A(t_0)+A(t_0)^\top\dbV(t_0)+C(t_0)^\top P_1(t_0,t_0)C(t_0) \\
&\qq+Q(t_0,t_0)+P_2(t_0)^\top M(t_0,t_0) P_2(t_0)+ C(t_0)^\top P_2(t_0)^\top N(t_0,t_0) P_2(t_0) C(t_0) \Big)\\
&\q -P_2(t_0)^\top G_2(t_0) \big(   \widehat{A}(t_0)  +\widehat{C}(t_0)P_2(t_0) +\widehat{D}(t_0) P_2(t_0)C(t_0) \big)\\
&\q 	 -\big(  \widehat{A}(t_0)  +\widehat{C}(t_0)P_2(t_0) +\widehat{D}(t_0) P_2(t_0)C(t_0)\big)^\top G_2(t_0) P_2(t_0),
\end{align*}
where the last inequality is deduced from the fact
\begin{align*}
\Theta(t_0)^\top \big( R(t_0,t_0)+ D(t_0)^\top P_1(t_0,t_0)D(t_0)+D(t_0)^\top P_2(t_0)^\top N(t_0,t_0) P_2(t_0)D(t_0) \big)  \Theta(t_0)  \geq 0.
\end{align*}
Then we have\vspace{-3mm}
\begin{align*}
&\dbV(t)- \big(G_1(T)+H^\top G_2(T) H)= \int_{t}^{T} -\dot{\dbV}(s)ds\\
&\leq \int_{t}^{T} \Big(   \dbV(s)A(s)+A(s)^\top\dbV(s)+C(s)^\top P_1(s,s)C(s) +Q(s,s)\\
&\qq \qq +P_2(s)^\top M(s,s) P_2(s)+ C(s)^\top P_2(s)^\top N(s,s) P_2(s) C(s) \\
&\qq \qq+ P_2(s)^\top G_2(s) \big(   \widehat{A}(s)  +\widehat{C}(s)P_2(s) +\widehat{D}(s) P_2(s)C(s) \big)\\
&\qq \qq+ \big( \widehat{A}(s)  +\widehat{C}(s)P_2(s) +\widehat{D}(s) P_2(s)C(s)\big)^\top G_2(s) P_2(s)	\Big)ds.
\end{align*}
This together with \eqref{pr-eq9} implies that
\begin{align*}
|\dbV(t)|&\leq \big(|H|^2+1\big)\widehat{\cG}+ 	\int_{t}^{T} \Big[  \( 2|A(s)|+|C(s)|^2\)|\dbV(s)|  + |\widehat{Q}|+ \widehat{\cG}^2|\widehat{A}(s)|^2 \\
& \qq + \( 1+ |\widehat{M}| +|\widehat{N}||C(s)|^2+ 2\widehat{\cG}|\widehat{C}(s)|+ 2 \widehat{\cG}|C(s)||\widehat{D}(s)|  \)
|P_2(s)|^2\Big]ds\\
&\leq \big(|H|^2+1\big)\widehat{\cG}+ 	\int_{t}^{T} \Big\{   \Big[ 2|A(s)|+|C(s)|^2  + \frac{n}{\delta} \Big( 1+ |\widehat{M}| +|\widehat{N}||C(s)|^2 \\
&\qq + 2\widehat{\cG}|\widehat{C}(s)|+ 2 \widehat{\cG}|C(s)||\widehat{D}(s)|  \Big) \Big]\dbV(s) + |\widehat{Q}|+ \widehat{\cG}^2|\widehat{A}(s)|^2 \Big\}ds,
\end{align*}
where   $\widehat{\cG}=\max\{|\widehat{G}_1|,|\widehat{G}_2|\}$, and $\delta$, $\widehat{Q}$, $\widehat{M}$, $\widehat{N}$ are given in Assumptions \ref{H2}--\ref{H3}. This, together with Gronwall's inequality, yields that
\begin{align*}\ns\ds
|\dbV(t)|\!\leq\! \Big[ \widehat{\cG}\!+\!\widehat{\cG}|H|^2\!+\! \int_{0}^{T} \(|\widehat{Q}|\!+\! \widehat{\cG}^2|\widehat{A}(s)|^2\) ds  \Big] e^{\int_{0}^{T}\big[ 2|A(s)|+|C(s)|^2 \! +\! \frac{n}{\delta} \big( 1+ |\widehat{M}| +|\widehat{N}||C(s)|^2 \!+\! 2\widehat{\cG}|\widehat{C}(s)|\!+ 2 \widehat{\cG}|C(s)||\widehat{D}(s)|  \big)\big]ds },
\end{align*}	
for any $t\in [0,T]$. The desired result is proved.
\end{proof}

Consequently, we can also get the following priori estimates for $\Theta(\cd)$.
\begin{corollary}\label{cro1}
Under Assumptions \ref{H2}-\ref{H5}, suppose the ERE \eqref{ERE} admits a solution $(P_1(\cd,\cd),P_2(\cd))\in C(\Delta^*[0,T];\dbS^n_+)\times C([0,T];\dbR^{m \times n})$.
Then the matrix-valued function $\Theta(\cd)$ defined in \eqref{Theta} satisfies the uniform norm estimate:\vspace{-2mm}
\begin{equation}
\|\Theta(\cd)\|_{C([0,T];\dbR^{k\times n})}\leq \cC^*,
\end{equation}
where the constant $\cC^*$ is explicitly given by
\begin{align}\ns\ds
&\cC^*\=  \frac{\cC^*_{\dbV}}{\delta}\Big( \|B(\cd)\|_{C([0,T];\dbR^{n\times k} )}+\|C(\cd)\|_{C([0,T];\dbR^{n\times n} )}\|D(\cd)\|_{C([0,T];\dbR^{n\times k} )}  \Big)  \nonumber\\
&\qq   + \frac{n \cC^*_{\dbV}}{\delta^2} \(  1+ |\widehat{N}| \|C(\cd)\|_{C([0,T];\dbR^{n\times n} )}\|D(\cd)\|_{C([0,T];\dbR^{n\times k} )} + \widehat{\cG} \|B(\cd)\|_{C([0,T];\dbR^{n\times k} )} \nonumber \\
& \qq \qq\q +\widehat{\cG}\|D(\cd)\|_{C([0,T];\dbR^{n\times k} )}\|\widehat{D}(\cd)\|_{C([0,T];\dbR^{m\times k} )}\)
+\frac{1}{\delta}   \widehat{\cG}^2 \|\widehat{B}(\cd)\|^2_{C([0,T];\dbR^{m\times k} )}.\label{de-C*}
\end{align}
Here  $\widehat{\cG}=\max\{|\widehat{G}_1|,|\widehat{G}_2|\}$, $\cC^*_{\dbV}$ is defined in \eqref{def-CV},  and $\delta$, $\widehat{N}$ are given in Assumptions \ref{H2}-\ref{H3}.
\end{corollary}

\begin{proof}
Combining the estimate from Proposition \ref{pri-est-Pi}, the governing equation \eqref{pr-eq9}, and the definition of $\Theta(\cd)$ in \eqref{Theta}, we derive the following pointwise bound for the matrix norm of $\Theta(\cd)$:
\begin{align*}
|\Theta(t)|&\leq   \frac{1}{\delta} \Big[ |P_1(t,t)| \Big( \|B(\cd)\|_{C([0,T];\dbR^{n\times k} )}  +\|C(\cd)\|_{C([0,T];\dbR^{n\times n} )}\|D(\cd)\|_{C([0,T];\dbR^{n\times k} )} \Big)  \\
& \qq+ |P_2(t)|^2 \Big( 1+ |\widehat{N}| \|C(\cd)\|_{C([0,T];\dbR^{n\times n} )}\|D(\cd)\|_{C([0,T];\dbR^{n\times k} )}  + \widehat{\cG} \|B(\cd)\|_{C([0,T];\dbR^{n\times k} )}  \\
& \qq +\widehat{\cG}\|D(\cd)\|_{C([0,T];\dbR^{n\times k} )}\|\widehat{D}(\cd)\|_{C([0,T];\dbR^{m\times k} )}   \Big) + \widehat{\cG}^2   \|\widehat{B}(\cd)\|^2_{C([0,T];\dbR^{m\times k} )} \Big] \leq \cC^*.
\end{align*}
\end{proof}

\section{Well-posedness of the ERE \eqref{ERE} with smooth coefficients}\label{sec5}

\subsection{Local solvability}

In this subsection, we prove the following local solvability result for the  ERE  \eqref{ERE}.

\begin{theorem}\label{th1}(Local Well-Posedness).
Under Assumptions \ref{H2}-\ref{H5}, there exists a $\k\in (0,T]$ such that the  ERE \eqref{ERE} admits a unique solution $(P_1(\cd,\cd),P_2(\cd))\in C(\Delta^*[T-\k,T];\dbS^n_+)\times C([T-\k,T];\dbR^{m\times n})$.
\end{theorem}

\begin{proof}
By Lemma \ref{lm3.5}, it suffices to establish the well-posedness for IES \eqref{Inte}, which is divided into four steps.

\ss

{\bf Step 1.}   By virtue of Lemma \ref{lm3.4}, we introduce the mapping $\Gamma: C([T-\k,T];\dbR^{k\times n})\to C([T-\k,T];\dbR^{k\times n})$  for some $\k\in (0,T]$ (to be determined later) as follows:
\begin{align}
\Gamma [ \Theta(\cd)](t)&=\! -\!\big[R(t,t) \!+ \! D(t)^\top \!\big( {P}_1(t,t)\! + \!  P_2(t)^\top N(t,t) P_2(t) \big)D(t)\big]^{-1} \nonumber\\  \ns \ds
&\q \times  \big[ B(t)^\top {P}_1(t,t)     +  D(t)^\top\big({P}_1(t,t)  +  P_2(t)^\top N(t,t) P_2(t) \big)C(t)   \nonumber\\  \ns \ds
&\qq + \big(\widehat{B}(t)^\top\! +  {B}(t)^\top  P_{2}(t)^\top   +  D(t)^\top  P_{2}(t)^\top   \widehat{D}(t)^\top\big ) G_2(t) P_{2}(t) \big],   \q t\in [T-\k,T],
\end{align}
where $P_1(\cd,\cd)$ and $P_2(\cd)$ are determined by equations \eqref{theta-P1P2} on $[T\!-\k,T]$ for each given $\Theta(\cd)\!\in\!\! C([T-\k,T];\dbR^{k\times n})$.

Let\vspace{-3mm}
\begin{align}\label{de-ball-T}
\dbB_T(2\cC^*)\!\=\! \Big\{ \!M \! \in \dbR^{k\times n}  \,\Big|&\, \Big|M -\!\big[R(T,T) \!+ \! D(T)^\top \!\big( {G}_1(T)\! + \!  H^\top N(T,T) H \big)D(T)\big]^{-1} \nonumber\\  \ns \ds
& \times  \big[ B(T)^\top {G}_1(T)     +  D(T)^\top\big({G}_1(T)  +  H^\top N(T,T) H \big)C(T)   \nonumber\\  \ns \ds
&+ \big(\widehat{B}(T)^\top\! +  {B}(T)^\top  H^\top   +  D(T)^\top  H^\top   \widehat{D}(T)^\top\big ) G_2(T) H \big] \Big| \leq 2 \cC^*\Big\},
\end{align}
where $\cC^*$ is defined in \eqref{de-C*}.

We claim that there exists a $\k\in (0,T]$ to ensure the following property:

{\bf Property $(P)_T$}: 	For any $\Theta(\cd)\in C([T-\k,T];\dbR^{k\times n})$ satisfying $\Theta(t)\in \dbB_T(2\cC^*)$ for $t\in [T-\k,T]$, we have
\begin{equation}\label{pro-p}
\Gamma[\Theta(\cd)](t)\in  \dbB_T(2\cC^*),\q \forall t\in [T-\k,T].
\end{equation}
We shall prove this claim in this step.

Following Corollary \ref{cro1}, we first observe that
\begin{align*}
&\Big|\big[R(T,T) \!+ \! D(T)^\top \!\big( {G}_1(T)\! + \!  H^\top N(T,T) H \big)D(T)\big]^{-1}  \big[ B(T)^\top {G}_1(T)  \!   + \! D(T)^\top\big({G}_1(T)  \!+\!  H^\top N(T,T) H \big)C(T)   \nonumber\\  \ns \ds
&\q+ \big(\widehat{B}(T)^\top +  {B}(T)^\top  H^\top   +  D(T)^\top  H^\top   \widehat{D}(T)^\top\big ) G_2(T) H \big] \Big| \leq  \cC^*.
\end{align*}
Consequently, for any $\Theta(\cd)$ satisfying the hypothesis of {\bf Property $(P)_T$}, we obtain the uniform bound:
\begin{equation}\label{pr-th1-eq1}
\|\Theta(\cd)\|_{C([T-\k,T];\dbR^{k\times n})}\leq 3\cC^*.
\end{equation}
Substituting such $\Theta(\cd)$ into \eqref{theta-P1P2}, we derive that
\begin{align*}
&\Big| \Gamma[\Theta(\cd)](t)   -\!\big[R(T,T) \!+ \! D(T)^\top \!\big( {G}_1(T)\! + \!  H^\top N(T,T) H \big)D(T)\big]^{-1} \nonumber\\  \ns \ds
& \times  \big[ B(T)^\top {G}_1(T)     +  D(T)^\top\big({G}_1(T)  +  H^\top N(T,T) H \big)C(T)   \nonumber\\  \ns \ds
&\qq + \big(\widehat{B}(T)^\top\! +  {B}(T)^\top  H^\top   +  D(T)^\top  H^\top   \widehat{D}(T)^\top\big ) G_2(T) H \big] \Big|
\leq \dbL_1+\dbL_2,
\end{align*}
where\vspace{-2mm}
{\small
\begin{align}
&\dbL_1=\frac{1}{\delta^2} \Big( |R(t,t)\!-\!R(T,T)|\! +\! |D(t)\!-\!D(T)| |P_1(t,t)| |D(t)|\!+\! |D(T)||P_1(t,t)\!-\!P_1(T,T)||D(t)| \nonumber\\
&\qq\q \,+ |D(T)||P_1(T,T)||D(t)\!-\!D(T)|\!+\! |D(t)\!-\!D(T)||P_2(t)|^2|\widehat{N}||D(t)| \nonumber \\
&\qq \q\,+|D(T)||P_2(t)\!-\!P_2(T)||\widehat{N}||P_2(t)||D(t)|\!+\! |D(T)||P_2(T)||N(t,t)\!-\!N(T,T)||P_2(t)||D(t)| \nonumber\\
&\qq \q \,+ |D(T)||P_2(T)||\widehat{N}||P_2(t)\!-\!P_2(T)||D(t)|\!+\!  |D(T)||P_2(T)|^2|\widehat{N}||D(t)\!-\!D(T)|\Big) \nonumber\\
&\qq \times \Big(  \big( |B(t)| +|D(t)||C(t)|\big)|P_1(t,t)|+ \big(1+ |\widehat{N}||D(t)||C(t)|+ \widehat{\cG}|B(t)| + \widehat{\cG}|D(t)||\widehat{D}(t)| \big)  \nonumber\\
&\qq\qq	  \times|P_2(t)|^2 +  \widehat{\cG}^2|\widehat{B}(t)|^2   \Big),\label{I1}
\end{align}}
and\vspace{-3mm}
{\small
\begin{align}
&\dbL_2=\frac{1}{\delta} \Big( |B(t)-B(T)||P_1(t,t)|+ |B(T)||P_1(t,t)-P_1(T,T)|+|D(t)-D(T)||P_1(t,t)||C(t)| \nonumber\\
&\qq \q +  |D(T)||P_1(t,t)-P_1(T,T)||C(t)|+  |D(T)||P_1(T,T)||C(t)-C(T)| \nonumber\\
& \qq \q + |D(t)-D(T)||P_2(t)|^2|\widehat{N}||C(t)|+   |D(T)||P_2(t)-P_2(T)||\widehat{N}||P_2(t)||C(t)| \nonumber\\
&\qq\q +    |D(T)||P_2(T)||N(t,t)-N(T,T)||P_2(t)||C(t)|+     |D(T)||P_2(T)||\widehat{N}||P_2(t)-P_2(T)||C(t)| \nonumber\\
&\qq\q + |D(T)||P_2(T)||\widehat{N}||P_2(T)||C(t)-C(T)|+ \widehat{\cG}|\widehat{B}(t)-\widehat{B}(T)||P_2(t)| \nonumber\\
&\qq\q + |\widehat{B}(T)||G_2(t)-G_2(T)||P_2(t)|+ \widehat{\cG}|\widehat{B}(T)||P_2(t)-P_2(T)| \nonumber\\
&\qq\q + \widehat{\cG}|B(t)-B(T)||P_2(t)|^2+\widehat{\cG}|B(T)||P_2(t)-P_2(T)||P_2(t)| \nonumber\\
&\qq\q+ |B(T)||P_2(T)||G_2(t)-G_2(T)||P_2(t)| +\widehat{\cG}|B(T)||P_2(T)||P_2(t)-P_2(T)| \nonumber\\
&\qq\q + \widehat{\cG}|D(t)-D(T)||P_2(t)|^2|\widehat{D}(t)|+  \widehat{\cG}|D(T)||P_2(t)-P_2(T)||\widehat{D}(t)||P_2(t)| \nonumber\\
&\qq\q +  \widehat{\cG} |D(T)||P_2(T)||\widehat{D}(t)-\widehat{D}(T)||P_2(t)|+   |D(T)||P_2(T)||\widehat{D}(T)||G_2(t)-G_2(T)||P_2(t)| \nonumber\\
&\qq\q + \widehat{\cG} |D(T)||P_2(T)||\widehat{D}(T)||P_2(t)-P_2(T)| \Big).\label{I2}
\end{align}}
Therefore, to establish {\bf Property $(P)_T$}, it suffices to estimate the quantities $|P_1(t,t)|,|P_2(t)|,|P_1(t,t)-P_1(T,T)|, |P_2(t)-P_2(T)|$  for $t\in [T-\k,T]$. From Lemma \ref{lm3.4} (particularly equations \eqref{lm3.1-eq1}--\eqref{lm3.1-eq3}) and \eqref{pr-th1-eq1}, we obtain the following uniform bounds:
\begin{align}
&\sup_{\k\in (0,T]}\sup_{t\in[T-\k,T]}\(\mE \Big[\sup_{s\in[t,T]}|\Phi(t,s)|^2 \Big]+|P_1(t,t)|\)+\sup_{\k\in (0,T]} \|P_2(\cd)\|_{C([T-\k,T];\dbR^{m\times n})} \leq \cC(\cC^*),\label{pr-th1-eq2}
\\&|P_2(t)-P_2(T)|\leq \cC(\cC^*) |t-T|,\label{pr-th1-eq5}\\
&|P_1(t,t)-P_1(T,T)|\leq\cC(\cC^*)\big(|t-T|^{1\over 2}+|G_1(t)-G_1(T)|\big).\label{pr-th1-eq6}
\end{align}
Here $ \cC(\cC^*)$ is a general constant depending on $\cC^*$ and system parameters, but independent of $\k\in (0,T]$.

Finally, combining estimates \eqref{pr-th1-eq2}--\eqref{pr-th1-eq6} with the equicontinuity of the coefficient functions $B(\cd)$, $C(\cd)$, $D(\cd)$, $\widehat{B}(\cd)$, $\widehat{D}(\cd)$, $G_1(\cd)$, $G_2(\cd)$, $R(\cd,\cd)$ and $N(\cd,\cd)$,  we conclude that there exists a $\k\in (0,T]$, depending solely on system parameters, for which {\bf Property $(P)_T$} holds.

\ss

{\bf Step 2.} Fix $\k\in (0,T]$ with {\bf Property $(P)_T$}. Define $\Gamma_{[T-h,T]}: C([T-h,T];\dbR^{k\times n})\to C([T-h,T];\dbR^{k\times n})$ $(0<h\leq \k)$
to be the restriction of $\Gamma[\cd]$ on time interval $[T-h,T]$ as follows:
\begin{align}
\Gamma_{[T-h,T]} [ \Theta(\cd)](t) &=\! -\!\big[R(t,t) \!+ \! D(t)^\top \!\big( {P}_1(t,t)\! + \!  P_2(t)^\top N(t,t) P_2(t) \big)D(t)\big]^{-1} \nonumber\\  \ns \ds
&\q \times  \big[ B(t)^\top {P}_1(t,t)     +  D(t)^\top\big({P}_1(t,t)  +  P_2(t)^\top N(t,t) P_2(t) \big)C(t)   \nonumber\\  \ns \ds
&\qq + \big(\widehat{B}(t)^\top\! +  {B}(t)^\top  P_{2}(t)^\top   +  D(t)^\top  P_{2}(t)^\top   \widehat{D}(t)^\top\big ) G_2(t) P_{2}(t) \big], \q t\in [T-h,T],  \label{de-gama-h}
\end{align}
where $P_1(\cd,\cd)$ and $P_2(\cd)$ are determined by equation \eqref{theta-P1P2} on $[T-h,T]$  for the given $\Theta(\cd)\in C([T-h,T];\dbR^{k\times n})$.

In this step, we prove that $\Gamma_{[T-h,T]}[\cd]$ is contractive if $h$ is small enough.

Let $\Theta_1(\cd),\Theta_2(\cd)\in C([T-h,T];\dbR^{k\times n})$ with   $\Theta_i(t)\in\dbB_T(2\cC^*)$ for all $t \in [T-h,T]$, $i = 1,2$. We immediately obtain the uniform bound:\vspace{-3mm}
\begin{align}
\|\Theta_i(\cd)\|_{C([T-h,T];\dbR^{k\times n})}\leq 3\cC^*,\q i=1,2.\label{pr-th1-eq7}
\end{align}
Associated with these functions, we obtain the corresponding solution triples $(\Phi^i(\cd,\cd),P^i_1(\cd,\cd),P^i_2(\cd))$ for $i = 1,2$, satisfying the following uniform estimates:
\begin{align}
\sup_{t\in[T-h,T]}\mE \Big[\sup_{s\in[t,T]}|\Phi^i(t,s)|^2 \Big] + \|P^i_2(\cd)\|_{C([T-h,T];\dbR^{m\times n})}+ \sup_{t\in [T-h,T]}|P^i_1(t,t)|\leq \cC(\cC^*),\label{pr-th1-eq8}
\end{align}
where $\cC(\cC^*)$ is a general constant depending on $\cC^*$ and system parameters, but independent of $h\in (0,\k]$.
Moreover, by the standard estimates of SDEs, for $0< h\leq \k$, we have
\begin{align}
\sup_{t\in[T-h,T]}\mE \Big[ \sup_{s\in [t,T]} |\Phi^1(t,s)-\Phi^2(t,s)|^2 \Big]&\leq  \cC(\cC^*) \int_{T-h}^{T} |\Theta_1(s)-\Theta_2(s)|^2ds.\label{pr-th1-eq11}
\end{align}
From \eqref{pr-th1-eq7}, \eqref{pr-th1-eq8} and Gronwall's inequality, we obtain that\vspace{-2mm}
\begin{align}
\big|P^1_2(t)-P_2^2(t)\big|&\leq  \cC(\cC^*) \int_{T-h}^{T} |\Theta_1(s)-\Theta_2(s)|ds, \q t\in [T-h,T].\label{pr-th1-eq12}
\end{align}
From \eqref{pr-th1-eq7}--\eqref{pr-th1-eq11}, we   get that
\begin{align}
&\q \big| P_1^1(t,t)-P_1^2(t,t) \big|\nonumber\\
&\leq \mE  \big| \Phi^1(t,T)^\top G_1(t)\Phi^1(t,T)- \Phi^2(t,T)^\top G_1(t)\Phi^2(t,T)\big| \nonumber\\
&\q + \mE\int_{t}^{T} \Big(    \big|\Phi^1(t,s)^\top  Q(t,s)    \Phi^1(t,s) - \Phi^2(t,s)^\top  Q(t,s)    \Phi^2(t,s) \big| \nonumber \\
& \qq \qq  +   \big|\Phi^1(t,s)^\top \Theta_1   ^\top R(t,s)\Theta_1       \Phi^1(t,s) - \Phi^2(t,s)^\top  \Theta_2   ^\top R(t,s)\Theta_2      \Phi^2(t,s) \big| \nonumber\\
& \qq \qq  +   \big|\Phi^1(t,s)^\top {P^1_2}  ^\top M(t,s) P_2^1       \Phi^1(t,s) - \Phi^2(t,s)^\top  {P^2_2}  ^\top M(t,s) P_2^2      \Phi^2(t,s) \big| \nonumber\\
& \qq \qq  +   \big|\Phi^1(t,s)^\top C_{\Theta_1}    ^\top  {P^1_2} ^\top N(t,s)  {P}_2^1    C_{\Theta_1}        \Phi^1(t,s) - \Phi^2(t,s)^\top  C_{\Theta_2}    ^\top  {P^2_2} ^\top N(t,s)  {P}_2^2    C_{\Theta_2}       \Phi^2(t,s) \big| \Big) ds \nonumber\\
&\leq \cC(\cC^*) \(\int_{T-h}^{T} |\Theta_1(s)-\Theta_2(s)|^2 ds \)^{1/2},\hspace{16em} t\in [T-h,T]. \label{pr-th1-eq13}
\end{align}
Here we emphasize that the estimates \eqref{pr-th1-eq8}--\eqref{pr-th1-eq13} are uniform for $h\in (0,\k]$.
Recalling the definition of $\Gamma_{[T-h,T]}[\cd]$ in \eqref{de-gama-h},
by \eqref{pr-th1-eq8} and \eqref{pr-th1-eq12}--\eqref{pr-th1-eq13}, we have\vspace{-2mm}
\begin{align}
&\sup_{t\in [T-h,T]} \big| \Gamma_{[T-h,T]}[\Theta_1(\cd)](t)- \Gamma_{[T-h,T]}[\Theta_2(\cd)](t) \big| \nonumber\\
&\leq \cC(\cC^*)  \( \sup_{t\in [T-h,T]}\big|P^1_1(t,t)-P^1_1(t,t)\big|+ \sup_{t\in [T-h,T]} \big|P^1_2(t)-P^1_2(t)\big|  \) \nonumber\\
&\leq \cC(\cC^*)  \(\int_{T-h}^{T} |\Theta_1(s)-\Theta_2(s)|^2 ds \)^{1/2} \nonumber\\
&\leq \cC(\cC^*) \sqrt{h} \sup_{s\in [T-h,T]} |\Theta_1(s)-\Theta_2(s)|.\label{pr-th1-eq14}
\end{align}
Thus, by selecting $h$ sufficiently small so that  $ \cC(\cC^*) \sqrt{h}\leq \frac{1}{2}$,
we conclude that $\Gamma_{[T-h,T]}[\cd]$  is indeed a contraction mapping on the specified function space.

\ss

{\bf Step 3.} Having established the contraction property in Step 2, we now prove the well-posedness of the IES \eqref{Inte}  on the interval $[T-h,T]$, where $h > 0$ is chosen as in Step 2.

Define a sequence of functions $\{\Theta_n(\cd)\}_{n\geq 0}$ recursively:\vspace{-2mm}
\begin{align*}
\Theta_0(t)=&
-\!\big[R(T,T) \!+ \! D(T)^\top \!\big( {G}_1(T)\! + \!  H^\top N(T,T) H \big)D(T)\big]^{-1} \nonumber\\  \ns \ds
& \times  \big[ B(T)^\top {G}_1(T)     +  D(T)^\top\big({G}_1(T)  +  H^\top N(T,T) H \big)C(T)   \nonumber\\  \ns \ds
&+ \big(\widehat{B}(T)^\top\! +  {B}(T)^\top  H^\top   +  D(T)^\top  H^\top   \widehat{D}(T)^\top\big ) G_2(T) H \big],& t\in [T-h,T],\\\ns\ds
\Theta_{n+1}(t)&=\Gamma_{[T-h,T]}[\Theta_n(\cd)](t),& t\in [T-h,T],\q  n\geq 0.
\end{align*}
By \eqref{pro-p}, each iteration remains in the desired ball: \vspace{-2mm}
\begin{align*}
\Theta_n(t)\in \dbB_T(2\cC^*),\q t\in [T-h,T],\q n\geq 0.
\end{align*}
The contraction estimate \eqref{pr-th1-eq14} yields \vspace{-2mm}
\begin{align*}
\|\Theta_{n+2}(\cd)-\Theta_{n+1}(\cd)\|_{C([T-h,T];\dbR^{k\times n})} \leq \frac{1}{2} \|\Theta_{n+1}(\cd)-\Theta_{n}(\cd)\|_{C([T-h,T];\dbR^{k\times n})}.
\end{align*}
This implies that $\{\Theta_n(\cd)\}_{n\geq 0}$ is a Cauchy sequence in $C([T-h,T];\dbR^{k\times n})$,
and hence converges uniformly to a limit function  $\Theta^*(\cd) \in C([T-h,T];\dbR^{k\times n})$.
The closedness of $\dbB_T(2\cC^*)$ ensures
\begin{align*}
\Theta^*(t)\in \dbB_T(2\cC^*),\q t\in [T-h,T].
\end{align*}
Then\vspace{-4mm}
\begin{align}
&\qq \big\|\G_{[T-h,T]}[\Theta^*(\cd)] - \Theta^*(\cd) \big\|_{C([T-h,T];\,\dbR^{k\times n})} \nonumber\\
&\q= \lim_{k \to \infty}  \big\|\G_{[T-h,T]}[\Theta^*(\cd)] - \Theta^*(\cd)
-\G_{[T-h,T]}[\Theta_k(\cd)]+\Theta_{k+1}(\cd)  \big\|_{C([T-h,T];\,\dbR^{k\times n})} \nonumber\\
&\q\leq\lim_{k \to \infty}  \Big(   \frac{1}{2} 	\|\Theta_{k}(\cd)-	\Theta^*(\cd)\|_{C([T-h,T];\,\dbR^{k\times n})}
+\|\Theta_{k+1}(\cd)-	\Theta^*(\cd)\|_{C([T-h,T];\,\dbR^{k\times n})} \Big ) \nonumber\\
&\q=0.\label{pr-th1-eq15}
\end{align}
Therefore, $\Theta^*(\cd)$ is a fixed point of $\Gamma_{[T-h,T]}[\cd]$.

Substituting $\Theta^*(\cd)$ into \eqref{theta-P1P2} yields a solution $(\{P^*_1(t,t);t\in [T-h,T]\},P^*_2(\cd))$ of IES \eqref{Inte} on $[T-h,T]$.

On the other hand, Corollary \ref{cro1} shows that $\Theta(\cd)$ corresponding to any solution of  IES \eqref{Inte} must satisfy\vspace{-1mm}
\begin{align*}
\Theta(t)\in \dbB_T(2\cC^*),\q t\in [T-h,T].
\end{align*}
This combined with the fact that $\Gamma_{[T-h,T]}[\cd]$ is a contraction mapping  for
any $\Theta(\cd)\in C([T-h,T];\dbR^{k\times n})$ valued in $\dbB_T(2\cC^*)$,
leads to the uniqueness of the solution of IES \eqref{Inte} over $[T-h,T]$.

\ss

{\bf Step 4}. In this step, we complete the proof using an inductive argument on the time interval. Without loss of generality, we assume $2h\leq \k$ and establish the well-posedness of IES \eqref{Inte} on the extended interval $[T-2h,T]$.

Define the sequence $\{\Theta_n(\cd)\}_{n\geq 0}$ for the mapping $\Gamma_{[T-2h,T]}[\cd]$ as follows:
\begin{align*}
&\Theta_0(t)=\Theta^*(T-h)\chi_{[T-2h,T-h]}(t)+\Theta^*(t)\chi_{(T-h,T]}(t),\q t\in [T-2h,T],\\ \ns\ds
&\Theta_{n+1}(t)=\Gamma_{[T-2h,T]}[\Theta_n(\cd)](t),\q t\in [T-2h,T],\q n\geq 0.
\end{align*}
Since $\Theta_0(t)$ takes values in $\dbB_T(2\cC^*)$ for $t\in[T-2h,T]$, from {\bf Property $(P)_T$}, we have
\begin{align*}
\Theta_n(t)\in \dbB_T(2\cC^*),\q t\in [T-2h,T],\q n\geq 0.
\end{align*}
The consistency condition \eqref{pr-th1-eq15} from Step 3 yields:
\begin{align}
\Theta_n(t)=\Theta^*(t),\q t\in [T-h,T],\q n\geq 0.\label{pr-th1-eq16}
\end{align}
Using the uniform bounds \eqref{pr-th1-eq8}--\eqref{pr-th1-eq13} and following similar arguments as in \eqref{pr-th1-eq14}, we obtain
\begin{align*}
\big|\Theta_{n+2}(t)-	\Theta_{n+1}(t)\big|
&= \big|\G_{[T-2h,T]}[\Theta_{n+1}(\cd)](t) - \G_{[T-2h,T]}[\Theta_{n}(\cd)](t)  \big|\\
&\leq 	\cC(\cC^*)  \(\int_{T-2h}^{T}  |\Theta_{n+1}(s)-{\Theta_{n}}(s)|^2ds\)^{1/2}\\
&\leq 	\cC(\cC^*)  \(\int_{T-2h}^{T-h}  |\Theta_{n+1}(s)-{\Theta_{n}}(s)|^2ds\)^{1/2},\q t\in [T-2h,T],
\end{align*}
where $\cC(\cC^*)$ is the same as that in \eqref{pr-th1-eq14} and the final inequality follows from \eqref{pr-th1-eq16}.
This leads to the contraction estimate: \vspace{-1mm}
\begin{align*}
\|\Theta_{n+2}(\cd)-\Theta_{n+1}(\cd)\|_{C([T-2h,T];\,\dbR^{k\times n})}
&\leq\cC(\cC^*) \sqrt{h}\|\Theta_{n+1}(\cd)-\Theta_{n}(\cd)\|_{C([T-2h,T];\,\dbR^{k\times n})}\\
&\leq  \frac{1}{2}\|\Theta_{n+1}(\cd)-	\Theta_{n}(\cd)\|_{C([T-2h,T];\,\dbR^{k\times n})}.
\end{align*}
Therefore, following the same methodology as in Step 3:
we can prove the well-posedness of IES \eqref{Inte} over $[T-2h,T]$.
By iterating this procedure, we extend the well-posedness result to the entire interval $[T-\k,T]$, completing the proof.
\end{proof}
\subsection{Global solvability}

In this subsection, we prove the following global solvability result for the  ERE  \eqref{ERE}.

\begin{theorem} \label{th2}(Global Well-Posedness).
Under Assumptions (H2)--(H5), the  ERE  \eqref{ERE} admits a unique solution $(P_1(\cd,\cd),P_2(\cd))\in C(\Delta^*[0,T];\dbS^n_+)\times C([0,T];\dbR^{m\times n})$.
\end{theorem}

\begin{proof}
By Lemma \ref{lm3.5}, it suffices to establish the well-posedness of the IES \eqref{Inte} over the interval $[0,T]$.	

First, we define the mapping $\Gamma : C([T-\k-\e,T];\dbR^{k\times n}) \to C([T-\k-\e,T];\dbR^{k\times n}) $ in the following
(where $\k$ is as in Theorem \ref{th1} and $\e$ will be chosen later): 	
For any $\Theta(\cd)\in C([T-\k-\e,T];\dbR^{k\times n})$,\vspace{-1mm}
\begin{align}
\Gamma [ \Theta(\cd)](t)&=\! -\!\big[R(t,t) \!+ \! D(t)^\top \!\big( {P}_1(t,t)\! + \!  P_2(t)^\top N(t,t) P_2(t) \big)D(t)\big]^{-1} \nonumber\\  \ns \ds
&\q \times  \big[ B(t)^\top {P}_1(t,t)     +  D(t)^\top\big({P}_1(t,t)  +  P_2(t)^\top N(t,t) P_2(t) \big)C(t)   \nonumber\\  \ns \ds
&\qq + \big(\widehat{B}(t)^\top\! +  {B}(t)^\top  P_{2}(t)^\top   +  D(t)^\top  P_{2}(t)^\top   \widehat{D}(t)^\top\big ) G_2(t) P_{2}(t) \big],   \q t\in [T-\k-\e,T],
\end{align}
where $P_1(\cd,\cd)$ and $P_2(\cd)$ are given by \eqref{theta-P1P2} (with the time interval adjusted to $[T-\k-\e,T]$).

Next, we introduce the notation\vspace{-1mm}
\begin{align}\label{de-ball-T-k}
\dbB_{T-\k}(2\cC^*)\!\=\! \big\{ \!\Theta \! \in \dbR^{k\times n}  \,\big|&\, \big|\Theta  - \Theta^*(T-\k)  \big| \leq 2 \cC^*\big\},
\end{align}
where $\cC^*$ is defined in \eqref{de-C*} and $\Theta^*(\cd)$ is the unique fixed point determined in Step 4 of Theorem \ref{th1}.

We claim that, by choosing $\e=\k$, the following property holds:

{\bf Property $(P)_{T-\k}$:} 	For any $\Theta(\cd)\in C([T-2\k,T-\k];\dbR^{k\times n})$ satisfying $\Theta(T-\k)=\Theta^*(T-\k)$
and $\Theta(t)\in \dbB_{T-\k}(2\cC^*)$ for $t\in [T-2\k,T-\k]$, we have
\begin{equation}\label{pro-p-T-k}
\Gamma[\Theta(\cd)\chi_{[T-2\k,T-\k]}(\cd)+ \Theta^*(\cd)\chi_{(T-k,T]}(\cd)   ](t)\in  \dbB_{T-k}(2\cC^*),\q \forall t\in [T-2\k,T-\k].
\end{equation}

Now we prove this claim. Fix $\Theta(\cd)\in C([T-2\k,T-\k];\dbR^{k\times n})$ as above. By Theorem \ref{th1} and Corollary \ref{cro1}, we have $|\Theta^*(t)|\leq \cC^*$ for $t\in [T-\k,T]$.
Therefore,\vspace{-1mm}
\begin{align}
| \Theta(t)\chi_{[T-2\k,T-\k]}(t)+ \Theta^*(t)\chi_{(T-k,T]}(t)  |\leq 3 \cC^*,\q t\in [T-2\k,T].\label{pr-th2-eq1}
\end{align}
Using Theorem \ref{th1} again, we have\vspace{-1mm}
\begin{align*}
\Gamma[\Theta(\cd)\chi_{[T-2\k,T-\k]}(\cd)+ \Theta^*(\cd)\chi_{(T-k,T]}(\cd)   ](t)= \Theta^*(t),\q t\in [T-\k,T],
\end{align*}
Consequently,\vspace{-1mm}
\begin{align}
&\q \big| \Gamma[ \Theta(\cd)\chi_{[T-2\k,T-\k]}(\cd)+ \Theta^*(\cd)\chi_{(T-k,T]}(\cd)  ](t) -\Theta^*(T-\k) \big|\nonumber\\
&=\big| \Gamma[ \Theta(\cd)\chi_{[T-2\k,T-\k]}(\cd) \!+ \! \Theta^*(\cd)\chi_{(T-k,T]}(\cd)  ](t)  \!- \! \Gamma[ \Theta(\cd)\chi_{[T-2\k,T-\k]}(\cd) \!+ \! \Theta^*(\cd)\chi_{(T-k,T]}(\cd)  ](T-\k)  \big|\label{pr-th2-eq1.1}\\
&\leq \dbL_1+\dbL_2,\nonumber
\end{align}
where $\dbL_1$ and $\dbL_2$ are the same as those in \eqref{I1}--\eqref{I2} (with $T$ replaced by $T-\k$).
Note that with \eqref{pr-th2-eq1}, by Lemma \ref{lm3.4},  we can similarly obtain the following estimates:\vspace{-1mm}
\begin{align}
&  \sup_{t\in [T-2\k,T]}|P_1(t,t)|+ \|P_2(\cd)\|_{C([T-2\k,T];\dbR^{m\times n})} \leq  \cC(\cC^*),\label{pr-th2-eq1.2}\\
&  \qq |P_2(t)-P_2(T-\k)|\leq \cC(\cC^*) |t-(T-\k)|,\label{pr-th2-eq1.3}
\end{align}
and\vspace{-1mm}
\begin{align}
&|P_1(t,t)-P_1(T-\k,T-\k)|\nonumber\\ &\leq\cC(\cC^*)\(|t\!-(T\!-\!\k)|^{1\over 2}+|G_1(t)\!-\!G_1(T\!-\!\k)|\) +\cC(\cC^*)\int_{T-\k}^T\! \Big(|Q(t,s)\!-\!Q(T\!-\!\k,s)|+|R(t,s)\!-\!R(T\!-\!\k,s)|\nonumber \\
&   \q  + |M(t,s)-M(T-\k,s)|+|N(t,s)-N(T-\k,s)|	\Big)ds,\label{pr-th2-eq1.4}
\end{align}
where the above $\cC(\cC^*)$ is the same as that in \eqref{pr-th1-eq2}--\eqref{pr-th1-eq6}.

Combining \eqref{pr-th2-eq1.1}--\eqref{pr-th2-eq1.4} with  the equicontinuity and uniform continuity of the coefficients yields the claim.

By iterating the above argument (following Steps 2--4 of Theorem \ref{th1}), we extend the well-posedness of \eqref{Inte} piecewise to $[T-2\k,T]$, and subsequently to the full interval $[0,T]$.
\end{proof}

\section{Well-posedness of the ERE \eqref{ERE} with non-smooth coefficients}\label{sec6}

In this section, we prove Theorem \ref{Th2}.
\begin{proof}[Proof of Theorem \ref{Th2}]
By Theorem \ref{th2}, establishing the well-posedness of the IES \eqref{Inte} reduces to constructing suitable approximations of the system coefficients.

To this end, we select uniformly bounded sequences {\small$\{A_n(\cd),C_n(\cd)\}_{n\geq 1 } \!\subset \! C^{\infty}([0,T]; \dbR^{n\times n})$,
$\{B_n(\cd), D_n(\cd)\}_{n\geq 1 } \!\!\subset\!\! C^{\infty} \!([0,T]; \dbR^{n\times k})$,
$\{\widehat{C}_n(\cd), \!\widehat{D}_n(\cd)\}_{n\geq 1 } \!\!\subset\!\! C^{\infty} \!([0,T]; \dbR^{m\times m})$,
$\{\widehat{A}_n(\cd)\}_{n\geq 1 }  \!\!\subset\! \! C^{\infty} \!([0,T]; \dbR^{m\times n})$,
$\{\widehat{B}_n(\cd)\}_{n\geq 1 } \!\!\subset\!\! C^{\infty} \!([0,T]; \dbR^{m\times k})$}
such that\vspace{-1mm}
\begin{align*}
&\lim_{n\to \infty} \big(\|A_n- A \|_{L^2(0,T; \dbR^{n\times n})}\!+\!\|C_n -C \|_{L^2(0,T; \dbR^{n\times n})}\!+\!\|B_n - B \|_{L^2(0,T; \dbR^{n\times k})}+\|D_n -D \|_{L^2(0,T; \dbR^{n\times k})}\big) =0,  \\
& \lim_{n\to \infty} \big(\|\widehat{A}_n -\widehat{A} \|_{L^2(0,T; \dbR^{m\times n})}\!+\!\|\widehat{C}_n\! -\widehat{C} \|_{L^2(0,T; \dbR^{m\times m})}\!+\!\|\widehat{B}_n -\widehat{B} \|_{L^2(0,T; \dbR^{m\times k})}\!+\!\|\widehat{D}_n -\widehat{D}  \|_{L^2(0,T; \dbR^{m\times m})}\big)=0.	
\end{align*}
For the weighting matrices, Lemma \ref{lm3.3} guarantees the existence of approximating sequences:  $\{G_1^n(\cd)\}_{n\geq 1} \subset  C^{\infty}([0,T];\,\dbS^n)$,
$\{G_2^n(\cd)\}_{n\geq 1} \subset  C^{\infty}([0,T];\,\dbS^m)$,  $\{Q_n(\cd,\cd)\}_{n\geq 1}\subset C([0,T]^2;\, \dbS^{n})$,
$\{R_n(\cd,\cd)\}_{n\geq 1}\subset C([0,T]^2;\, \dbS^{k})$, and  $\{M_n(\cd,\cd), N_n(\cd,\cd)\}_{n\geq 1}\subset C([0,T]^2;\, \dbS^{m})$ such that
\begin{itemize}
\item  $ 0 \leq G_1^n(t)\leq G_1^n(r) \leq \widehat{G}_1,\q  \delta I \leq G_2^n(t)\leq G_2^n(r) \leq \widehat{G}_2,\q   0\leq t \leq r \leq T,$
\item  $  \lim\limits_{n\to \infty}\|G_1^n(\cd)-G_1(\cd)\|_{C([0,T];\,\dbS^n)}=0,$\q   $\lim\limits_{n\to \infty}\|G_2^n(\cd)-G_2(\cd)\|_{C([0,T];\,\dbS^m)}=0,$
\item  $ \text{for any } 0\leq s \leq T, \,  Q_n(t,s), R_n(t,s), M_n(t,s) \text{ and } N_n(t,s)   \text{ are infinitely differentiable for $t$,}$
\item  $0 \leq Q_n(t,s)\leq Q_n(r,s) \leq \widehat{Q},  \q  0 \leq M_n(t,s)\leq M_n(r,s) \leq \widehat{M},$ \\
$0 \leq N_n(t,s)\leq N_n(r,s) \leq \widehat{N}, \q\delta I \leq R_n(t,s)\leq R_n(r,s) \leq \widehat{R},\q  0\leq t \leq r \leq s \leq T,$
\item  $\ds\!\!\!\lim\limits_{n \to \infty} \!\!\int_{0}^{T}\!\!\!\! \sup\limits_{t\in [0,s]} \!\!\big(\big|  Q_n(t,s) \! - \! Q(t,s)  \big|^2 \! \!+\!\big|  M_n(t,s) \! - \! M(t,s)  \big|^2\! \!+\!\big|  N_n(t,s) \! -\!  N(t,s)  \big|^2\!\! +\! \big|R_n(t,s) \! -\!  R(t,s)  \big|^2\big) ds \!= \!0$.
\end{itemize}
For each $n$, consider the IES \eqref{Inte} with coefficients $\{A_n,B_n, C_n,D_n,\widehat{A}_n,\widehat{B}_n,\widehat{C}_n,\widehat{D}_n,Q_n,R_n,M_n,N_n,G_1^n,\\G_2^n\}_{n\geq 0}$. By Theorem \ref{th2}, there exists a unique solution $(P^n_1(\cd,\cd),P_2^n(\cd))$ satisfying the uniform estimates:
\begin{align*}
&\sup_{n\geq 0} \|\Theta_n(\cd)\|_{C([0,T]; \dbR^{n \times n})}\leq \cC(\cC^*),\q
\sup_{n\geq 0} \sup_{t\in [0,T]} \mE \Big[  \sup_{s\in [t,T]}  |\Phi_n(t,s)|^2\Big]\leq  \cC(\cC^*),\\\ns
&\sup_{n\geq 0} \sup_{t\in [0,T]} |P_1^n(t,t)|\leq  \cC(\cC^*),\q
\sup_{n \geq 0} \|P_2^n(\cd) \|_{C([0,T];\dbR^{m\times n})}\leq  \cC(\cC^*),
\end{align*}
where $ \cC(\cC^*)$ depends only on $\cC^*$ and the system parameters. Since\vspace{-1mm}
\begin{align*}
&|\Theta_n(t)-\Theta_m(t)|\\ &\leq  \cC(\cC^*) \Big(  |P_1^n(t,t)-P_1^m(t,t)|\! +\! |P_2^n(t)\!-\!P_2^m(t)|\!+\!  |R_n(t,t)-R_m(t,t)|+  |N_n(t,t)-N_m(t,t)| \! + \! |B_n(t)-B_m(t)|\\
&\qq \q+  |C_n(t)\!-\!C_m(t)|\!+\! |D_n(t)-D_m(t)| +|\widehat{B}_n(t)-\widehat{B}_m(t)|\!+\!|\widehat{D}_n(t)\!-\!\widehat{D}_m(t)|+ |G_2^n(t)\!-\!G_2^m(t)| \Big),
\end{align*}
by the standard estimates for SDEs, we have\vspace{-1mm}
\begin{align*}
&\mE\Big[ \sup_{s\in [t,T]} |\Phi_n(t,s)-\Phi_m(t,s)|^2 \Big]\nonumber\\ \ns
&\leq \cC(\cC^*) \int_{t}^{T}  \Big( |A_n-A_m|^2+ |B_n-B_m|^2+ |C_n-C_m|^2+ |D_n-D_m|^2  + |\Theta_n-\Theta_m|^2 \Big)ds\nonumber\\ \ns
&\leq \cC(\cC^*) \int_{t}^{T}  \Big(    |P_1^n(s,s)-P_1^m(s,s)|^2 + |P_2^n-P_2^m|^2+    |A_n-A_m|^2+ |B_n-B_m|^2 \nonumber\\\ns
& \qq \qq \qq  + |C_n-C_m|^2+ |D_n-D_m|^2 +|\widehat{B}_n-\widehat{B}_m|^2+|\widehat{D}_n-\widehat{D}_m|^2+ |G_2^n-G_2^m|^2\nonumber \\\ns
&\qq\qq\qq \qq    +    |R_n(s,s)-R_m(s,s)|^2+  |N_n(s,s)-N_m(s,s)|^2  \Big)ds.
\end{align*}
Combining these with Gronwall's inequality, we obtain \vspace{-1mm}
\begin{align}
& \q|P_1^n(t,t)-P_1^m(t,t)|^2  \nonumber\\
&\leq  \cC(\cC^*)  \Big[   \|G_1^n(\cd)-G_1^m(\cd)\|^2_{C([0,T];\dbS^n)}  +     \|G_2^n(\cd)-G_2^m(\cd)\|^2_{C([0,T];\dbS^m)}  \nonumber \\
& \qq \qq +   \int_{0}^{T} \Big( \sup_{t\in [0,s]} |Q_n(t,s)-Q_m(t,s)|^2+   \sup_{t\in [0,s]} |R_n(t,s)-R_m(t,s)|^2  \nonumber\\
&  \qq \qq \qq \q +  \sup_{t\in [0,s]} |M_n(t,s)-M_m(t,s)|^2+  \sup_{t\in [0,s]} |N_n(t,s)-N_m(t,s)|^2 \Big)ds  \nonumber \\
&\qq \qq + \int_{t}^{T} \Big( |P_1^n(s,s)-P_1^m(s,s)|^2+ |P_2^n-P_2^m|^2+|A_n-A_m|^2+ |B_n-B_m|^2  \nonumber\\
&\qq \qq \qq \qq +  |C_n-C_m|^2+ |D_n-D_m|^2+ |\widehat{B}_n-\widehat{B}_m|^2+  |\widehat{D}_n-\widehat{D}_m|^2  \Big)   ds \Big].\label{pr-pr5.1-eq2}
\end{align}
On the other hand, for $P_2^n(\cd)$ and $P_2^m(\cd)$, we have\vspace{-1mm}
\begin{align}
|P_2^n(t)-P_2^m(t)|^2&\leq \cC(\cC^*)   \int_{t}^{T} \Big( |P_1^n(s,s)-P_1^m(s,s)|^2+ |P_2^n-P_2^m|^2+|A_n-A_m|^2 \nonumber\\
&\qq \qq \qq  + |B_n-B_m|^2   +  |C_n-C_m|^2+ |D_n-D_m|^2+ |\widehat{A}_n-\widehat{A}_m|^2   \nonumber\\
&\qq \qq \qq \qq    + |\widehat{B}_n-\widehat{B}_m|^2+  |\widehat{C}_n-\widehat{C}_m|^2+  |\widehat{D}_n-\widehat{D}_m|^2  \Big)   ds. \label{pr-pr5.1-eq3}
\end{align}
Putting \eqref{pr-pr5.1-eq2}--\eqref{pr-pr5.1-eq3} together, by Gr\"onwall's inequality, we have
\begin{align*}
&\sup_{t\in [0,T]}  \Big( |P_1^n(t,t)-P_1^m(t,t)|^2+ 	|P_2^n(t)-P_2^m(t)|^2 \Big)\\
&\leq   \cC(\cC^*)  \Big[   \|G_1^n(\cd)-G_1^m(\cd)\|^2_{C([0,T];\dbS^n)}  +     \|G_2^n(\cd)-G_2^m(\cd)\|^2_{C([0,T];\dbS^m)}   \\
& \qq \qq +   \int_{0}^{T} \Big( \sup_{t\in [0,s]} |Q_n(t,s)-Q_m(t,s)|^2+   \sup_{t\in [0,s]} |R_n(t,s)-R_m(t,s)|^2  \\
&  \qq \qq \qq \q +  \sup_{t\in [0,s]} |M_n(t,s)-M_m(t,s)|^2+  \sup_{t\in [0,s]} |N_n(t,s)-N_m(t,s)|^2 \Big)ds  \\
&\qq \qq + \int_{0}^{T} \Big( |A_n-A_m|^2+ |B_n-B_m|^2  +  |C_n-C_m|^2+ |D_n-D_m|^2\\
&\qq \qq \qq \q \, +|\widehat{A}_n-\widehat{A}_m|^2 + |\widehat{B}_n-\widehat{B}_m|^2+ |\widehat{C}_n-\widehat{C}_m|^2+  |\widehat{D}_n-\widehat{D}_m|^2  \Big)   ds \Big].
\end{align*}
Thus $\big\{\big(\{P_1^n(t,t);t\in [0,T]\}, P^n_2(\cd)\big)\big\}_{n\geq 0}$ is a  Cauchy sequence in $C([0,T];\dbS^n)\times C([0,T];\dbR^{m \times n})$.
Taking $n\to\infty$ yields a solution to the original IES \eqref{Inte}. Uniqueness follows directly from the estimates.
\end{proof}

\end{document}